\documentclass[a4paper,english,twoside,12pt]{smfart}

\usepackage{smfenum}
\usepackage{smfthm}
\usepackage{amssymb}
\usepackage{mathrsfs,color,eucal}
\input{xypic}

\usepackage[T1]{fontenc}
\usepackage{calligra}

\theoremstyle{plain}

\def\theequation{\fnsymbol{equation}}

\newcounter{intro}

\newtheorem{theoint}[intro]{Theorem}

\newtheorem*{main}{Main Theorem}

\def\CC{\mathbb{C}}

\def\ZZ{\mathbb{Z}} 

\def\A{{\rm A}}
\def\B{{\rm B}}
\def\C{{\rm C}}
\def\D{{\rm D}}
\def\E{{\rm E}}
\def\F{{\rm F}}
\def\G{{\rm G}}
\def\H{{\rm H}}
\def\I{{\rm I}}
\def\J{{\rm J}}
\def\K{{\rm K}}
\def\L{{\rm L}}
\def\M{{\rm M}}
\def\N{{\rm N}}
\def\P{{\rm P}}

\def\rS{{\rm S}}
\def\T{{\Theta}}
\def\U{{\rm U}}
\def\V{{\rm V}}
\def\W{{\rm W}}
\def\X{{\rm X}}
\def\Y{{\rm Y}}
\def\Z{{\rm Z}}

\def\AA{\mathfrak{A}}
\def\BB{\mathfrak{B}}

\def\KK{\mathfrak{K}}
\def\LL{{\rm L}}

\def\PP{\mathfrak{P}}
\def\PB{\mathfrak{P}}
\def\SS{\mathfrak S}

\def\Cc{\EuScript{C}}

\def\Gg{\mathscr{G}}
\def\Hh{\mathscr{H}}

\def\Ll{\mathscr{L}}
\def\Mm{\mathscr{M}}   

\def\Oo{\EuScript{O}}

\def\Rr{\mathfrak{R}}

\def\Vv{\mathscr{V}}
\def\Ww{\mathscr{W}}

\def\b{\beta}
\def\d{\delta}
\def\e{{\rm e}}

\def\g{\gamma}
\def\h{\varphi}
\def\ii{{\iota}}

\def\k{\kappa}
\def\l{\lambda}

\def\p{\mathfrak{p}}

\def\s{\sigma}
\def\t{\theta}
\def\v{\upsilon}
\def\w{\varpi}
\def\y{\text{{\rm\Large\calligra y}\,}}
\def\z{\zeta}

\def\>{\geqslant}
\def\<{\leqslant}

\def\Hom{{\rm Hom}}
\def\End{{\rm End}}
\def\Aut{{\rm Aut}}

\def\GL{{\rm GL}}
\def\Gal{{\rm Gal}}

\def\tr{{\rm tr}}

\def\Ind{{\rm Ind}}
\def\mult#1{{#1}^{\times}}

\def\Nrd{{\rm Nrd}}
\def\diag{{\rm diag}}
\def\Id{{\rm Id}}

\def\ss{\mathfrak s}
\def\Ga{\Gamma}
\def\La{\Lambda}

\def\Om{\Omega}

\def\cInd{\hbox{{\rm c-Ind}}}

\def\flit{\diamond}
\def\id{{\rm id}}
\def\bb{{\boldsymbol b}}
\def\ll{{\boldsymbol l}}

\def\labelenumi{{\rm(\roman{enumi})}}
\def\theenumi{\roman{enumi}}
\def\labelenumii{{\rm(\alph{enumii})}}
\def\theenumii{\alph{enumii}}

\author{V. S\'echerre}
\address{Laboratoire de Math\'ematiques de Versailles,
Universit\'e de Versailles-Saint Quentin en Yvelines, 
45 avenue des \'Etats-Unis \\ 
78035 Versailles Cedex \\
France} 
\email{Vincent.Secherre@math.uvsq.fr}

\author{S. Stevens}
\address{School of Mathematics, University of East Anglia, 
  Norwich NR4 7TJ, United Kingdom}
\email{Shaun.Stevens@uea.ac.uk}

\title[Semisimple types for $\GL_m(\D)$]{Smooth representations of $\GL_m(\D)$\\ VI~: semisimple types}

\begin{abstract}
We give a complete description of the category of smooth complex
representations of the multiplicative group of a central simple algebra 
over a locally compact nonarchimedean local
field. More precisely, for each inertial class in the Bernstein
spectrum, we construct a type and compute its Hecke algebra. The Hecke
algebras that arise are all naturally isomorphic to products of affine
Hecke algebras of type $A$. We also prove that, for cuspidal
classes, the simple type is unique up to conjugacy.
\end{abstract}

\thanks{
This work was supported by: EPSRC (grants GR/T21714/01 and EP/G001480/1), 
the British Council and the Partenariat Hubert Curien 
in the framework of the Alliance programme (number 19418YK), 
the Agence Nationale de la Recherche (ANR-08-BLAN-0259-01), 
the Universit\'e de la M\'editerran\'ee Aix-Marseille 2 and the Institut de 
Math\'ematiques de Luminy.
}

\begin{document}

\maketitle

\tableofcontents


\section*{Introduction}

In~\cite{BK1}, Bushnell and Kutzko described a general approach
to understanding the category of smooth (complex) representations of a
reductive $p$-adic group $\G$: the theory of \emph{types}. This is based on
the Bernstein decomposition of the category~\cite{Be} into
indecomposable full subcategories, indexed by pairs $(\L,\pi)$, with
$\L$ a Levi subgroup of $\G$ and $\pi$ an irreducible cuspidal
representation of $\L$, up to a certain equivalence relation. A
\emph{type} for such a subcategory is a pair $(\K,\rho)$, with $\K$ a
compact open subgroup of $\G$ and $\rho$ an irreducible representation
of $\K$, such that the irreducible representations in the subcategory
are precisely those irreducible representations of $\G$ which contain
$\rho$. In this case, there is an equivalence of categories between
the subcategory and the category of left modules over the spherical
Hecke algebra $\Hh(\G,\rho)$. Thus one can classify the representations
of $\G$ by first constructing a type for each subcategory, and then
computing the spherical Hecke algebras.

This programme has been completed for general linear groups over a
$p$-adic field (Bushnell--Kutzko~\cite{BK,BKsemi}), for special linear
groups over a $p$-adic field (Bushnell--Kutzko~\cite{BKSL1,BKSL2},
Goldberg--Roche~\cite{GR1,GR2}) up to the computation of a two-cocycle in the
description of the Hecke algebra, and for three-dimensional unitary
groups in odd residual characteristic (Blasco~\cite{Bla}). In this paper,
following previous work in~\cite{VS1,VS2,VS3,SS4,BSS}, we complete the
programme for inner forms of general linear groups. The Hecke algebras
that arise are all naturally isomorphic to products of affine
Hecke algebras of type $A$.

\medskip

Let $\D$ be a division algebra over a locally compact nonarchimedean
local field $\F$, and let $\G=\GL_m(\D)$, with $m$ a positive
integer; we will also  think of $\G$ as the group of  automorphisms of a right
$\D$-vector space $\V$.  
Denote by $\Rr(\G)$ the category of smooth complex
representations of $\G$. In order to describe our results more precisely, we begin by recalling the Bernstein
decomposition~\cite{Be}, in the language of~\cite{BK1}.
For $\L$ a Levi subgroup of $\G$, denote by $\X(\L)$ the
complex torus of \emph{unramified characters} of $\L$: that is, smooth
homomorphisms $\L\to\CC^\times$ which are trivial on all compact
subgroups of $\L$. For $\pi$ an irreducible cuspidal representation of
$\L$, we write $[\L,\pi]_\G$ for the \emph{$\G$-inertial equivalence
class} of $(\L,\pi)$: that is, the set of pairs $(\L',\pi')$, consisting of
a Levi subgroup $\L'$ and an irreducible cuspidal representation
$\pi'$ of $\L'$, such that $(\L,\pi)$ and $(\L',\pi'\otimes\chi')$ are
$\G$-conjugate, for some unramified character $\chi'\in\X(\L')$.
We denote by $\BB(\G)$ the set of $\G$-inertial equivalence classes of
pairs $(\L,\pi)$ as above.

To $\ss=[\L,\pi]_\G$ an inertial equivalence class we associate a
full subcategory $\Rr^\ss(\G)$ of $\Rr(\G)$, whose objects are those
representations all of whose subquotients have cuspidal support in
$\ss$. Then the Bernstein decomposition says that
$$
\Rr(\G) = \prod_{\ss\in\BB(\G)}\Rr^\ss(\G).
$$

Bushnell--Kutzko's theory of types~\cite{BK1} is a strategy to understand the subcategories in this decomposition.
For $\ss\in\BB(\G)$, an \emph{$\ss$-type} is a pair $(\K,\rho)$, with
$\K$ a compact open subgroup of $\G$ and $(\rho,\Ww)$ an irreducible
smooth representation of $\K$, such that the irreducible representations in
$\Rr^\ss(\G)$ are precisely those irreducible representations of $\G$
which contain $\rho$. In that case, there is an equivalence of
categories
\begin{eqnarray*}
\Rr^\ss(\G) &\to& \Hh(\G,\rho)\hbox{-Mod} \\
\Vv &\mapsto& \Hom_\K(\Ww,\Vv),
\end{eqnarray*}
where $\Hh(\G,\rho)$ is the convolution algebra of compactly supported
$\End_\CC(\check{\Ww})$-valued function $f$ of $\G$ which satisfy
$f(hgk)=\check\rho(h)f(g)\check\rho(k)$, for $h,k\in\K$, $g\in\G$ (the
spherical Hecke algebra).

Our main result is the following:

\begin{main}
\label{thm:main}
Let $\ss\in\BB(\G)$. There exists an $\ss$-type $(\K,\rho)$ in $\G$,
such that
$$
\Hh(\G,\rho)\cong\bigotimes_{i=1}^l \Hh(r_i,q_{\F}^{f_i}),
$$
for some positive integers $l$ and $r_i,f_i$, for $1\le i\le l$.
\end{main}

Here $q_\F$ denotes the cardinality of the residue field of $\F$, and
$\Hh(n,q)$ is the affine Hecke algebra of type $\tilde A_{n-1}$ with
parameter $q$. In particular, the category of modules of such algebras
is completely understood, through the work of Kazhdan--Lusztig~\cite{KL}. 
The values of $l,r_i,f_i$ can be described as follows.

To $\pi$ an irreducible cuspidal representation of $\G=\GL_m(\D)$, we may
associate two invariants. First there is the \emph{torsion number $n(\pi)$},
the (finite) number of unramified characters $\chi$ of $\G$ such that
$\pi\simeq\pi\chi$. Second, we have the 
\emph{reducibility number $s(\pi)$}: writing $\tilde\G=\GL_{2m}(\D)$ and 
$\tilde\P$ for the
standard parabolic subgroup of $\tilde\G$ with Levi subgroup
$\G\times\G$, it is the unique positive real number such that the
induced representation
$\Ind_{\tilde\P}^{\tilde\G}\pi\otimes\pi\nu^{s(\pi)}$ (with respect to normalized 
parabolic induction) is reducible,
where $\nu(g)=|\Nrd(g)|_\F$, $\Nrd$ denotes the reduced norm of $\A$
over $\F$, and $|\cdot|_\F$ is the normalized absolute value on $\F$
(see~\cite[Theorem~4.6]{VSU0}). 

Now suppose $\ss=[\L,\pi]_\G$ is an inertial equivalence class. 
The Levi subgroup $\L$ is the stabilizer of some decomposition 
$\V=\bigoplus_{j=1}^r\V^j$ into subspaces, which gives an identification of 
$\L$ with $\prod_{j=1}^r\GL_{m_j}(\D)$, where $m_j=\dim_\D\V^j$. 
We can then write $\pi=\bigotimes_{j=1}^r\pi_j$, for
$\pi_j$ an irreducible cuspidal representation of $\G_j=\GL_{m_j}(\D)$.
We define an equivalence relation on $\{1,\ldots,r\}$ by
\begin{equation}\label{eqn:equivrel}
j\sim k\ \Longleftrightarrow\ m_j=m_k\hbox{ and }[\G_j,\pi_j]_{\G_j}=[\G_k,\pi_k]_{\G_k},
\end{equation}
where we have identified $\G_j$ with $\G_k$ whenever $m_j=m_k$.  
We may, and do, assume that $\pi_j=\pi_k$ whenever $j\sim k$, since this does 
not change the inertial class $\ss$. 
Denote by $S_1,\ldots,S_l$ the equivalence classes.  
For $i=1,\ldots,l$, we set
$$
r_i=\#S_i,\qquad f_i= n(\pi_j)s(\pi_j),\hbox{ for any }j\in S_i.
$$
These are then the parameters appearing in the Hecke algebras of the Main Theorem.\ref{thm:main}

\medskip

We now describe in more detail the construction of the types.  
The starting point is the construction of the irreducible cuspidal 
representations of $\G$, which was achieved in~\cite{VS3,SS4}. 
In~\cite{VS3}, generalizing the work of Bushnell--Kutzko for
$\GL_n(\F)$~\cite{BK}, the first author constructed a set of pairs
$(\J,\l)$ called \emph{simple types}. Amongst these are the
\emph{maximal simple types}, which give rise to cuspidal
representations: if $(\J,\l)$ is a maximal simple type then $\l$
extends to a representation $\tilde\l$ of its normalizer $\tilde\J$
and the compactly-induced representation
$\cInd_{\tilde\J}^{\G}\tilde\l$ is irreducible and cuspidal.
The main result of~\cite{SS4} is that all irreducible cuspidal
representations of $\G$ arise in this way. Here we prove more:

\begin{theoint}\label{thm:cuspidal}
Let $\pi$ be an irreducible cuspidal representation of $\G$ and
$\ss=[\G,\pi]_\G$. There is a maximal simple type $(\J,\l)$ contained
in $\pi$, and any such is an $\ss$-type. Moreover, it is unique up to
$\G$-conjugacy: that is, if $(\J_i,\l_i)$, for $i=1,2$, are maximal
simple types contained in $\pi$ then there exists $g\in\G$ such that
${}^g\J_1=\J_2$ and ${}^g\l_1=\l_2$. 
\end{theoint}

The uniqueness (up to conjugacy) is proved in
Corollary~\ref{cor:maxsimpleconj}. We remark that, in the case of
$\GL_n(\F)$, the uniqueness here follows from an ``intertwining
implies conjugacy'' result for simple types. In the case of $\G$ there
is no such result for two reasons: Firstly there is an extra invariant
of a simple type, the \emph{embedding type} (see
paragraph~\ref{SS.embed}) and simple types with different embedding
types may intertwine but be non-conjugate. Secondly there is an action
of a galois group on simple types, and any two types in the same orbit
will intertwine but may be non-conjugate; this phenomenon arises
already for level zero representations in~\cite{GSZ}.
Nonetheless, the embedding type of a \emph{sound simple type} (see
\S\ref{S.IIC}) is determined by any irreducible representation
containing it. Moreover, in the case of maximal simple types (which
are always sound), the
galois action can be realized by conjugation by an element of the
normalizer of $\J$, and the uniqueness follows from this. 

\medskip

We turn to the case of a general $\G$-inertial equivalence class
$\ss=[\L,\pi]_\G$, for which we have the corresponding cuspidal $\L$-inertial
equivalence class $\ss_\L=[\L,\pi]_\L$. In~\cite{BK1}, Bushnell and
Kutzko present a general framework for constructing an $\ss$-type from
an $\ss_\L$-type $(\J_\L,\l_\L)$: the theory of \emph{covers}. We do
not recall precisely the definition of a cover here, only that it
should be a pair $(\J_\G,\l_\G)$ which has an Iwahori decomposition
with respect to
any parabolic subgroup with Levi component $\L$, such that the Hecke
algebra $\Hh(\G,\l_\G)$ contains a suitable invertible element. If one
has such a cover $(\J_\G,\l_\G)$ then it is an $\ss$-type.

The normalizer $\N_\G(\L)$ acts on
$\BB(\L)$ by conjugation and there is a Levi subgroup $\M$ of $\G$
which is minimal for the property of containing the
$\N_\G(\L)$-stabilizer of $\ss_\L$. Then we also have an $\M$-inertial
equivalence class $\ss_\M=[\L,\pi]_\M$. The strategy now is first to
construct a cover $(\J_\M,\l_\M)$ of $(\J_\L,\l_\L)$, and then a cover
$(\J_\G,\l_\G)$ of $(\J_\M,\l_\M)$ -- by transitivity of covers, this
will give the required $\ss$-type. 

\medskip

In our situation, we do indeed have an $\ss_\L$-type: Writing $\L=\prod_{j=1}^r\GL_{m_j}(\D)$ and $\pi=\bigotimes_{j=1}^r\pi_j$ as above, for
$\pi_j$ an irreducible cuspidal representation of
$\G_j=\GL_{m_j}(\D)$, there is a maximal simple type $(\J_j,\l_j)$ which is a $[\G_j,\pi_j]_{\G_j}$-type, by Theorem~\ref{thm:cuspidal}.
Then, putting 
$$
\J_\L=\prod_{j=1}^r \J_j,\qquad \l_\L=\bigotimes_{j=1}^r \l_j,
$$
it is clear that $(\J_\L,\l_\L)$ is an $\ss_\L$-type. The Levi subgroup $\M$ is then that defined by the equivalence relation~\eqref{eqn:equivrel}: it is the stabilizer of the decomposition $\V=\bigoplus_{i=1}^l \Y^i$, where $\Y^i=\bigoplus_{j\in S_i}\V^j$.

The first case to consider is when $\M=\G$, 
that is, the case $l=1$ of the Main Theorem. 
In this situation we have the following result, which
summarizes~\cite[Proposition~5.5,~Th\'eor\`eme~4.6]{VS3}
and~\cite[Th\'eor\`eme~5.23]{SS4}:

\begin{theoint} 
\label{thm:simplecover}
Let $\pi_0$ be an irreducible cuspidal representation of $\G_0=\GL_{m_0}(\D)$,
with $m=rm_0$, and let $(\J_0,\l_0)$ be a maximal simple type which is a 
$[\G_0,\pi_0]_{\G_0}$-type. Let $\L={\G_0}^r$ be a Levi subgroup of $\G$ with
irreducible cuspidal representation $\pi={\pi_0}^{\otimes r}$, and put
$\ss=[\L,\pi]_\G$. Put $\J_\L={\J_0}^r$ and $\l_\L={\l_0}^{\otimes r}$.
Then there is an $\ss$-type $(\J_\G,\l_\G)$ which is a cover of $(\J_\L,\l_\L)$, and
$$
\Hh(\G,\l_\G)\simeq \Hh(r,q_\F^f),
$$
where $f=n(\pi_0)s(\pi_0)$. Moreover, there is a simple type $(\J,\l)$ in $\G$ such that $\l=\Ind_{\J_\G}^{\J}\l_\G$. 
\end{theoint}

Now we turn to the general case of arbitrary $\M$. In order to describe the 
covering process, we need to recall some detail of the structure of simple 
types.  

Let $(\J,\l)$ be a simple type contained in an irreducible representation 
$\pi$ of $\G$.  There is a particular filtration of pro-$p$ subgroups 
$\{\H^{t+1}:t\ge 0\}$ of $\J$ such that $\l$ restricts to a multiple of a 
character $\theta^{(t)}$ on $\H^{t+1}$, and 
$\theta^{(0)}\ |\ \H^{t+1}=\theta^{(t)}$. These characters are called~\emph{simple
  characters of  level $t$}. Simple characters  were the main  object of study
of~\cite{VS1,BSS} and they exhibit remarkable functorial properties, as in the
case $\D=\F$~\cite{BK,BH}. In  particular, it is possible to  transfer them to
the multiplicative  group of other central simple  $\F$-algebras. A convenient
and  powerful way  to express  this is  in terms  of  \emph{endo-classes} (see
\cite{BH,BSS}  and~\S\ref{S.common}):   the  simple  character  $\theta^{(t)}$
determines  and   endo-class  $\Theta^{(t)}$,   which  depends  only   on  the
representation $\pi$. 

Now let $\pi=\bigotimes_{j=1}^r \pi_j$ be a cuspidal representation of $\L$ as above and denote by $\Theta_j^{(t)}$ the endo-class of level $t$ determined by $\pi_j$. (We are assuming a normalization of the index in the filtrations.)
For each integer $t\ge 0$, we define an equivalence relation on $\{1,\ldots,r\}$ by
$$
j \sim_t k \Longleftrightarrow \Theta_j^{(t)}\hbox{ is endo-equivalent to }\Theta_k^{(t)}.
$$ 
As for the equivalence relation~\eqref{eqn:equivrel}, this determines a Levi subgroup $\M_t$. Note that
$$
\M \subseteq \M_0;\qquad \M_t\subseteq\M_{t'}\hbox{, for }t\ge t';\qquad\hbox{and }\M_t=\G\hbox{, for sufficiently large }t.
$$
It is useful to extend the notation and put $\M_t=\M$ for $t<0$. Of course, although the Levi subgroups are indexed by an integer $t$, there are only finitely many in $\{\M_t:t\in\ZZ\}$.

Theorem~\ref{thm:simplecover} provides a cover $(\J_\M,\l_\M)$ of $(\J_\L,\l_\L)$ in $\M$ and the Main Theorem now follows from the transitivity of covers and:

\begin{theoint}\label{thm:coverhecke}
For $t\ge t'$, there are a cover $(\J_{\M_{t'}},\l_{\M_{t'}})$ of $(\J_{\M_t},\l_{\M_t})$ in $\G$ and a support-preserving isomorphism of Hecke algebras
$\Hh(\M_{t},\l_{\M_{t}})\simeq\Hh(\M_{t'},\l_{\M_{t'}})$.
\end{theoint}

As always in the theory of covers, the difficulty is in defining the groups 
$\J_{\M_t}$. In fact, many covers were already constructed in~\cite[\S4]{SS4} 
and we must show that we can put ourselves in the situation of \emph{loc.\ 
  cit.}. For this, we need to use the notion of a \emph{common approximation} 
of simple characters from~\cite{BKsemi}, which is essentially a 
reinterpretation of the notion of endo-class. 

\medskip

We end the introduction with a summary of the contents of each
section. Section~\ref{S.notation} consists of basic definitions, as
well as recalling a very useful technique from~\cite{BSS} for reducing
proofs to easier situations. Section~\ref{S.pairs} concerns simple strata and
pairs, while section~\ref{S.characters} concerns simple characters;
these are the technical heart of the paper, in particular the
translation principle
Theorem~\ref{thm:translation} which is needed to cope with the fact that a simple character may be defined relative to several inequivalent simple strata. 
Along the way, we prove a generalization of a conjecture in~\cite{BSS} on the embedding type of a simple character.
Section~\ref{S.common} concerns the 
relationship between endo-classes and common
  approximations. Section~\ref{S.simple} recalls basic results about
simple types but in the more general situation of lattice sequences which is needed later, and we prove the uniqueness results
in section~\ref{S.IIC}. Finally, the general construction of a cover is given
in section~\ref{S.Semisimple}.

Much of the material here is necessarily technical. A
reader who is already familiar with the situation (and common
notations) for the case $\D=\F$ and is interested only in seeing the main
results could probably get by reading only
sections~\ref{S.IIC} and~\ref{S.Semisimple}.


\section{Notation and preliminaries}\label{S.notation}

Let $\F$ be a nonarchimedean locally compact field. For $\K$ a finite
extension of $\F$, or more generally a division algebra over a finite
extension of $\F$, we denote by $\Oo_{\K}$ its ring of integers,  by
$\p_{\K}$ the maximal ideal of $\Oo_{\K}$ and by $k_\K$ its
residue field.

For $u$ a real number, we denote by $\lceil{u}\rceil$ the smallest integer 
which is greater than or equal to $u$, and by $\lfloor{u}\rfloor$ the greatest 
integer which is smaller than or equal to $u$, that is, its integer part.

All representations considered are smooth and complex.

\subsection{}
Let $\A$ be a simple central $\F$-algebra, and let $\V$ be a simple left 
$\A$-module. 
The algebra $\End_{\A}(\V)$ is an $\F$-division algebra, the opposite of 
which we denote by $\D$.
Considering $\V$ as a right $\D$-vector space, we have a canonical 
isomorphism of $\F$-algebras between $\A$ and $\End_{\D}(\V)$. 

\begin{defi}
An \emph{$\Oo_{\D}$-lattice sequence} on $\V$ is a map 
$$
\La:\ZZ \to \{\hbox{$\Oo_\D$-lattices of $\V$}\}
$$ 
which is decreasing (that is, $\La(k)\supseteq\La(k+1)$ for all
$k\in\ZZ$) and such that there exists a positive integer
$e=e(\La|\Oo_\D)$ satisfying $\La(k+e)=\La(k)\p_\D$, for all $k\in\ZZ$. 
This integer is called the \emph{$\Oo_\D$-period} of $\La$ over $\Oo_\D$.

If $\La(k)\supsetneq\La(k+1)$ for all $k\in\ZZ$, then the lattice
sequence $\La$ is said to be \emph{strict}.
\end{defi}

Associated with an $\Oo_{\D}$-lattice sequence $\La$ on $\V$, we have an
$\Oo_{\F}$-lattice sequence on $\A$ defined by:
$$
\PP_{k}(\La)=\{a\in\A\ |\ a\La_{i}\subseteq\La_{i+k},\ i\in\ZZ\},
\quad k\in\ZZ.
$$
The lattice $\AA(\La)=\PP_0(\La)$ is a hereditary $\Oo_\F$-order in $\A$, 
and $\mathfrak{P}(\La)=\PP_1(\La)$ is its Jacobson radical.
They depend only on the set $\{\La(k)\mid k\in\ZZ\}$.

\medskip

We denote by $\KK(\La)$ the \emph{$\mult\A$-normalizer} of $\La$, 
that is the subgroup of $\mult\A$ made of all elements 
$g\in\mult\A$ for which there is an integer $n\in\ZZ$ such that 
$g(\La(k))=\La(k+n)$ for all $k\in\ZZ$.
Given  $g\in\KK(\La)$, such an integer 
is unique: it is denoted 
$\v_{\La}(g)$ and called the \emph{$\La$-valuation} of $g$. 
This defines a group homomorphism $\v_{\La}$ from $\KK(\La)$ to $\ZZ$.
Its kernel, denoted $\U(\La)$, is the group of invertible elements 
of $\AA(\La)$. 
We set $\U_0(\La)=\U(\La)$ and, for $k\>1$, we set 
$\U_k(\La)=1+\PP_k(\La)$.

\subsection{}
Let $\E$ be a finite extension of $\F$ contained in $\A$. We denote by
$e(\E/\F)$ and $f(\E/\F)$ the ramification index and residue class
degree respectively.

An $\Oo_\D$-lattice sequence $\La$ on $\V$ is said to be 
\emph{$\E$-pure} if it is normalized by $\E^{\times}$.
The centralizer of $\E$ in $\A$, denoted $\B$, is a
simple central $\E$-algebra.
We fix a simple left $\B$-module $\V_\E$ and write $\D_\E$ 
for the algebra opposite to $\End_{\B}(\V_\E)$.
By~\cite[Th\'eor\`eme 1.4]{SS4} 
(see also~\cite[Theorem 1.3]{Br1}), 
given an $\E$-pure $\Oo_{\D}$-lattice sequence on $\V$, 
there is an $\Oo_{\D_\E}$-lattice sequence $\Ga$ 
on $\V_\E$ such that:
$$
\PP_{k}(\La)\cap\B=\PP_{k}(\Ga), \quad k\in\ZZ.
$$
It is unique up to translation of indices, and its
$\B^\times$-normalizer is $\KK(\La)\cap\B^{\times}$.

\begin{defi}
A \emph{stratum} in $\A$ is a quadruple $[\La,n,m,\b]$ made of an 
$\Oo_\D$-lattice sequence $\La$ on $\V$, 
two integers $m,n$ such that $0\<m\<n-1$ and an element 
$\b\in\PP_{-n}(\La)$.
\end{defi} 

For $i=1,2$, let $[\La,n,m,\b_i]$ be a stratum in $\A$.
We say these two strata are \emph{equivalent}
if $\b_2-\b_1\in\PP_{-m}(\La)$. 

\subsection{}
Given a stratum $[\La,n,m,\b]$ in $\A$, we denote by $\E$ the 
$\F$-algebra generated by $\b$. 
This stratum is said to be \emph{pure} if $\E$ is a field, 
if $\La$ is $\E$-pure and if $\v_{\La}(\b)=-n$.
Given a pure stratum $[\La,n,m,\b]$, we denote by $\B$ the centralizer 
of $\E$ in $\A$.
For $k\in\ZZ$, we set:
\begin{equation*}
\mathfrak{n}_k(\b,\La)=\{x\in\AA(\La)\ |\ \b x-x\b\in\PP_k(\La)\}.
\end{equation*}
The smallest integer $k\>\v_{\La}(\b)$ such that $\mathfrak{n}_{k+1}(\b,\La)$ 
is contained in $\AA(\La)\cap\B+\PP(\La)$ is called the 
\emph{critical exponent} of the stratum $[\La,n,m,\b]$, denoted 
$k_0(\b,\La)$.

\begin{defi}
The stratum $[\La,n,m,\b]$ is said to be \emph{simple} if it is pure
and if we have $m\<-k_0(\b,\La)-1$. 
\end{defi}

Given $n\> 0$ and $\La$ an $\Oo_\D$-lattice sequence, there is another
stratum which plays a very similar role to simple strata, namely the
\emph{null stratum} $[\La,n,n-1,0]$. In particular, a simple stratum
$[\La,n,n-1,\b]$ is equivalent to a null stratum if and only if
$\b\in\PP_{1-n}(\La)$. 

\subsection{} 
Let $[\La,n,m,\b]$ be a simple stratum in $\A$.
The \emph{affine class} of $\La$ is the set of all $\Oo_{\D}$-lattice 
sequences on $\V$ of the form:
$$
a\La+b:k\mapsto\La(\lceil(k-b)/a\rceil),
$$
with $a,b\in\ZZ$ and $a\>1$.
The period of $a\La+b$ is $a$ times the period $e(\La|\Oo_\D)$ of
$\La$. 
The \emph{affine class} of the stratum $[\La,n,m,\b]$ is the set of
all (simple) strata of the form
$$
[\La',n',m',\b],
$$
where $\La'=a\La+b$ is in the affine class of $\La$, $n'=an$ and $m'$
is any integer such that $\lfloor m'/a\rfloor=m$. 

In the course of the paper, there will be several objects associated
to a simple stratum $[\La,n,m,\b]$, in particular simple characters
(see~\S\ref{S.characters}). By a straightforward induction
(see~\cite[Lemma~2.2]{BSS}), these objects depend only on the affine
class of the stratum.

\subsection{} 
This article makes use of a number of results of Grabitz~\cite{G}
which are based on the following definition.

\begin{defi}
A simple stratum $[\La,n,m,\b]$ in $\A$ is \emph{sound} 
if $\La$ is strict, $\AA\cap\B$ is principal and 
$\KK(\AA)\cap\mult\B=\KK(\AA\cap\B)$, where $\AA=\PP_0(\La)$ is 
the hereditary $\Oo_\F$-order defined by $\La$.
\end{defi}

The condition on $\AA\cap\B$ forces $\AA$ to be a principal $\Oo_\F$-order.
In the split case, a simple stratum $[\La,n,m,\b]$ is sound if and
only if $\La$ is strict and $\AA$ is principal.

\subsection{}\label{SS:passtoddag}
In~\cite[\S2.7]{BSS}, we developed a process to reduce many proofs to
the case of sound strata, which we recall briefly here: Let
$[\La,n,m,\b]$ be a simple stratum in $\A$ and let $e$ denote the
period of $\La$ over $\Oo_{\D}$.
Write $\B$ for the centralizer of the field $\E=\F(\b)$ in $\A$,
fix a simple left $\B$-module $\V_{\E}$ and write $\D_{\E}$ for the 
$\E$-algebra opposite to the algebra of $\B$-endomorphisms of $\V_{\E}$.
Let $\Ga$ denote an $\Oo_{\D_{\E}}$-lattice sequence on $\V_{\E}$ 
such that $\PP_{k}(\La)\cap\B=\PP_{k}(\Ga)$ for $k\in\ZZ$, and let
$e'$ denote its period over $\Oo_{\D_{\E}}$. 
We fix an integer $l$ which is a multiple of $e$ and $e'$ and set: 
$$
\La^{\ddag}:k\mapsto\La(k)\oplus\La(k+1)\oplus\cdots\oplus\La(k+l-1),
$$
which is a strict $\Oo_{\D}$-lattice sequence on
$\V^{\ddag}=\V\oplus\dots\oplus\V$ ($l$ times).
Thus we can form the simple stratum $[\La^{\ddag},n,m,\b]$
in $\A^{\ddag}=\End_{\D}(\V^{\ddag})$, where $\b$ is the block
diagonal element $\diag(\b,\ldots,\b)\in
\A^l\subseteq\A^\ddag$. By~\cite[Lemma~2.17]{BSS}, the stratum
$[\La^{\ddag},n,m,\b]$ is sound.

As we have mentioned, there will be several objects associated
to a simple stratum $[\La,n,m,\b]$ through the course of the paper. 
If one identifies $\A$ with the
``$(1,1)$-block'' of $\A^\ddag$ and intersect (or
restrict) these objects for $[\La^{\ddag},n,m,\b]$ to $\A$ one
recovers the corresponding objects for $[\La,n,m,\b]$ (see, for
example,~\cite[Th\'eor\`eme~2.17]{SS4}). Using 
this, in several proofs we write something like: ``by passing to
$\La^\ddag$ we may assume we are in the sound case''
(Lemma~\ref{lem:cruxstrata}, Proposition~\ref{prop:derived},
Lemma~\ref{lem:crux})). By this we mean
that we may prove the result for the sound stratum
$[\La^{\ddag},n,m,\b]$ and then
deduce the result for $[\La,n,m,\b]$ by intersecting with $\A$. In
general, it is safe to do this provided there is no conjugation
involved in the statement. An example of this is given already
in~\cite[Theorem~4.16]{BSS}.


\section{Simple strata and simple pairs}\label{S.pairs}

\subsection{}\label{SS.embed}
We begin by recalling the definition of a \emph{type} of embedding,
from~\cite{BG,BSS}. 

We fix a simple central $\F$-algebra $\A$ and a simple left
$\A$-module $\V$, and denote by $\D$ the opposite algebra of
$\End_\A(\V)$. An \emph{embedding} in $\A$ is a pair $(\E,\La)$
consisting of a finite field extension $\E$ of $\F$ contained in $\A$
and an $\E$-pure $\Oo_\D$-lattice sequence $\La$ on $\V$. Given such a
pair, we denote by $\E^\flit$ the maximal finite unramified extension
of $\F$ which is contained in $\E$ and whose degree divides the
reduced degree of $\D$ over $\F$.

Two embeddings $(\E_i,\La_i)$ are \emph{equivalent} if
there is an element $g\in\G$ such that $\La_1$ is in the translation class
of $g\La_2$ and $\E_1^\flit=g\E_2^\flit g^{-1}$. An equivalence class
for this relation is called an \emph{embedding type} in $\A$.

\begin{lemm}\label{lem:embedtype} 
Let $(\E,\La)$ be an embedding and put
$e=e(\E/F)$, $f=f(\E/\F)$. Let $\E'$ be a finite field
extension of $\F$ such that $e(\E'/F)=e$ and
$f(\E'/F)=f$. Then there is an embedding
$\ii:\E'\hookrightarrow\A$ such that $(\ii(\E'),\La)$ is an embedding
with the same embedding type as $(\E,\La)$. 
\end{lemm}

\begin{proof} When $\La$ is strict, this
is~\cite[Corollary~3.16(ii)]{BG}. For the general case, we fix a
simple right $\E\otimes_\F\D$-module $\rS$ and put
$\A(\rS)=\End_\D(\rS)$. Let $\B$ be the commutant of $\E$ in $\A$,
and let $\D_\E$ be the commutant of $\E$ in $\A(\rS)$. We also fix a
decomposition $\V=\V^1\oplus\cdots\oplus\V^l$ into simple right
$\E\otimes_\F\D$-modules (which are all copies of $\rS$) such that
the lattice sequence $\La$ decomposes into the direct sum of the
$\La^j=\La\cap\V^j$, for $j\in\{1,\ldots,l\}$. From~\cite[\S1.3]{VS3},
after choosing identifications $\V^j\simeq\rS$ for each $j$, we have
an $\F$-algebra embedding $\ii:\A(\rS)\to\A$. 
Denote by $\SS$ the unique (up to translation) $\E$-pure strict $\Oo_\D$-lattice sequence on $\rS$. By the strict case, there is an embedding $\rho:\E'\hookrightarrow\A(\rS)$ such that $(\rho(\E'),\SS)$ is an embedding with the same embedding type as $(\E,\SS)$. (Indeed, by~\cite[Remark~2.12]{BSS}, any embedding $(\rho(\E'),\SS)$ has the same embedding type as $(\E,\SS)$.) By conjugating the embedding, we may assume $\rho(\E')^\flit=\E^\flit$. Then the embedding $\ii\circ\rho$ has the required property.
\end{proof}

\subsection{}
We recall the definitions of \emph{simple pair} and \emph{endo-equivalence} from~\cite[Definitions~1.4,~1.7]{BSS} (see also~\cite[Definition~1.5]{BH}):

\begin{defi} 
A \emph{simple pair} over $\F$ is a pair $(k,\b)$ consisting of a non-zero element $\b$ of some finite extension of $\F$ and an integer $0\<k\<-k_{\F}(\b)-1$. 
\end{defi}

Let $\A$ be a simple central $\F$-algebra and $\V$ be a simple left $\A$-module. A \emph{realization} of a simple pair $(k,\b)$ in $\A$
is a stratum in $\A$ of the form $[\La,n,m,\h(\b)]$ made of:
\begin{enumerate}
\item
a homomorphism $\h$ of $\F$-algebra from $\F(\b)$ to $\A$;
\item
an $\Oo_{\D}$-lattice sequence $\La$ on $\V$ 
normalized by the image of $\F(\b)^{\times}$ 
under $\h$;
\item
an integer $m$ such that
$\left\lfloor m/e_{\h(\b)}(\La)\right\rfloor=k$.
\end{enumerate}
The integer $-n$ is then the $\La$-valuation of $\h(\b)$.
By \cite[Proposition 2.25]{VS1} we have:
$$
k_0(\h(\b),\La)=e_{\h(\b)}(\La)k_{\F}(\b),
$$
which implies that any realization of a simple pair is a simple stratum.

\begin{defi} 
\begin{enumerate}
\item For $i=1,2$, let $(k_{i},\b_{i})$ be a simple pair over $\F$.
We say that these pairs are \textit{endo-equivalent}, denoted:
\begin{equation*}
(k_{1},\b_1)\thickapprox(k_{2},\b_2),
\end{equation*}
if $k_1=k_2$ and $[\F(\b_1):\F]=[\F(\b_2):\F]$, and if
there exists a simple central $\F$-algebra $\A$ together with 
realizations $[\La,n_i,m_i,\h_i(\b_i)]$ of $(k_i,\b_i)$ 
in $\A$, with $i=1,2$, which intertwine in $\A$.
\end{enumerate}
\end{defi}

This defines an equivalence relation on simple pairs, from the following Proposition:

\begin{prop}[{\cite[Propositions~1.7,~1.9]{BSS}}]
\label{prop:IICsimplepairs}
For $i=1,2$, let $(k,\b_{i})$ be a simple pair over $\F$, and suppose 
these pairs are endo-equivalent. Let $\A$ be a simple central $\F$-algebra and let 
$[\La,n_i,m_i,\h_i(\b_i)]$ be a real\-ization of 
$(k,\b_i)$ in $\A$, for $i=1,2$.
These strata then intertwine in $\A$.

Moreover, if $n_1=n_2$, $m_1=m_2$, and $(F[\h_i(\b_i)],\La)$ have the same embedding type then these strata are conjugate in $\KK(\La)$.
\end{prop}

\subsection{}
Let $[\La,n,m,b]$ be a stratum in a simple central $\F$-algebra $\A=\End_\D(\V)$ and let
$[\widetilde\La,n,m,b]$ be the induced stratum in the split central simple $\F$-algebra
$\widetilde\A=\End_\F(\V)$, where $\widetilde\La$ denotes the $\Oo_\F$-lattice sequence defined by $\La$. By \cite[Th\'eor\`eme 2.23]{VS1},
this latter stratum is simple if and only the first one is, and in this 
case they are realizations of the same simple pair over $\F$.

We fix a uniformizer $\w_\F$ of $\F$ and set 
$$
\y=\y(b,\La)=\varpi_\F^{n/g}b^{e/g},
$$
where $e=e(\La|\Oo_\F)$ and $g=\gcd(n,e)$. Set
$$
\overline \y = \overline \y(b,\La) = \y+\PP_1(\widetilde\La),
$$
considered as an element of
$\PP_0(\widetilde\La)/\PP_1(\widetilde\La)$.

\begin{defi} 
With notation as above:
\begin{enumerate}
\item the characteristic polynomial of $\overline \y$ (in $k_\F[X]$) is
called the \emph{characteristic polynomial of the stratum $[\La,n,m,b]$};
\item the minimum polynomial of $\overline \y$ (in $k_\F[X]$) is
called the \emph{minimum polynomial of the stratum $[\La,n,m,b]$}.
\end{enumerate}
\end{defi}

\begin{rema}\label{rem:charminpol}
\begin{enumerate}
\item Since $\overline \y$ depends only on the equivalence class
of the stratum $[\La,n,n-1,b]$ (and the choice of the
uniformizer~$\w_\F$), the same is true of the minimum and
characteristic polynomials.  
\item If $b$ normalizes $\La$ then, by~\cite[Lemma~2.1.9]{Br}, the
element $\y$ depends only on the strict lattice sequence whose image is the image of
$\La$; hence the same is true of the minimum and characteristic
polynomials.
\item The characteristic polynomial of a stratum may
also be computed as the reduction modulo $\p_\F$ of the characteristic
polynomial of $\y$ in $\widetilde\A$; of course, the same is
\emph{not} true for the minimum polynomial.
\end{enumerate}
\end{rema}

\begin{prop}\label{prop:simpleminpol}
Let $[\La,n,n-1,b]$ be a stratum in $\A$. Then
$[\La,n,n-1,b]$ is equivalent to a simple stratum if and only if its
minimum polynomial is irreducible and not~$X$. Moreover, if $[\La,n,n-1,\b]$ is a simple stratum equivalent to $[\La,n,n-1,b]$ then, writing $\E=\F[\b]$, we have
$$
\Oo_\E+\PP_1(\La) = \Oo_\F[\y(b,\La)]+\PP_1(\La).
$$
\end{prop}

\begin{proof} Note first that both conditions imply that the $b$ normalizes $\La$: if the minimum polynomial of $[\La,n,n-1,b]$ is irreducible and not $X$ then $\overline\y(b,\La)$ is invertible so $\y(b,\La)$ normalizes $\La$, whence so does $b$. Hence, using Remark~\ref{rem:charminpol}(ii), we may (and do) assume in the proof that $\La$ is strict. Also, the final assertion is clear since, by the minimality of $\b$, the element $\y(\b,\La)+\p_\E$ generates the extension $k_\E/k_\F$, and $\y(\b,\La)+\PP_1(\La)=\y(b,\La)+\PP_1(\La)$.

Suppose $[\La,n,n-1,b]$ is equivalent to a simple stratum
$[\La,n,n-1,\b]$ and put $\E=\F[\b]$. Then $[\widetilde\La,n,n-1,\b]$ is
also simple; in particular, $\overline
\y(b,\La)=\overline \y(\b,\La)$ is a non-zero element of $k_\E$ in
$\PP_0(\widetilde\La)/\PP_1(\widetilde\La)$ so has irreducible minimum
polynomial not $X$. 

For the converse, suppose $[\La,n,n-1,b]$ has irreducible minimum
polynomial $\overline f(X)\in k_\F[X]$ and put $\d=\deg(\overline
f(X))$. Since $[\La,n,n-1,b]$ has characteristic polynomial which is a
power of $\overline f(X)$, it is non-split fundamental. 
Now the proof follows that of~\cite[Theorem~3.2.1]{Br} and we only
sketch the difference so this proof should be read alongside
\emph{loc.\ cit.} -- in particular, we will refer to notations used
in the proofs there.

Following the ideas of~\cite[\S3]{Br}, we treat first the simpler case
when $\overline f(X)$ is also irreducible as an element of $k_\D[X]$
-- that is, when $d$ is coprime to $\d$. We fix $\L/\F$ a maximal
unramified subfield of $\D$, so that $k_\L=k_\D$.

Let $f(X)\in\Oo_\F[X]$ be any monic polynomial which reduces modulo $\p_\F$ to give $\overline f(X)$. We choose a matrix $\overline\gamma\in\M_\d(k_\F)$ with minimum polynomial $\overline f(X)$. In~\cite[Definition~3.2.4]{Br}, Broussous defines the notion of \emph{$\overline\g$-standard form}, which we will use here. Since $[\La,n,n-1,b]$ has characteristic polynomial which is a power of $\overline f(X)$, it is non-split fundamental and, by~\cite[Proposition~3.2.5]{Br} there is a $u\in\U(\La)$ such that $[\La,n,n-1,ubu^{-1}]$ is equivalent to a stratum in $\overline\g$-standard form. Since the property of being equivalent to a simple stratum is unchanged by conjugation in $\U(\La)$, we may as well assume $[\La,n,n-1,b]$ is itself in $\overline\g$-standard form.

Now we follow the proof of~\cite[Theorem~3.2.1]{Br} in this case. In \emph{op.~cit.}~p.221, an element $\b$ is defined and the proof of \emph{op.~cit.}~Proposition~3.2.8 shows that there is $u\in\U(\La)$ such that $[\La,n,n-1,\b]$ is equivalent to $[\La,n,n-1,ubu^{-1}]$. (More precisely,~\cite[Proposition~3.2.9]{Br} is applied to the $\Oo_\L$-order $\M_{n_i}(\Oo_\L)$, in the notation there.) Moreover, $\b$ is minimal over $\F$ by~\cite[Lemma~3.2.10]{Br} so $[\La,n,n-1,b]$ is equivalent to the simple stratum $[\La,n,n-1,u^{-1}\b u]$.

Finally suppose we are in the general case where $\overline f(X)$ is not irreducible in $k_\L[X]$ and we decompose $\overline f(X)=\overline p_0(X)\cdots\overline p_{s-1}(X)$ into irreducibles. Now we follow~\cite[\S3.3]{Br} to reduce to previous case. The proof is essentially identical (but easier) so we will not repeat it -- the only point is that, in~\cite[Proposition~3.3.2]{Br}, the orders $\mathfrak A_0$ and $\mathfrak B_0$ can be taken to be equal, by the case where $\overline f(X)$ is irreducible, and then the lattice sequences denoted $\Mm^1$ and $\Ll$ are equal, which implies that $\Mm$ is the lattice sequence here denoted $\La$ and, in the notation of~\cite[Lemma~3.3.10]{Br}, we have $\mathfrak A=\mathfrak A'$.
\end{proof}

For $j=1,\ldots,r$, let $[\La^j,n,n-1,\b_j]$ be a simple stratum in
$\A^j=\End_\D(\V^j)$ with $e(\La^j|\Oo_\F)=e$. 
Put $\V=\V^1\oplus\cdots\oplus\V^r$, and set
$\La=\La^1\oplus\cdots\oplus\La^r$, a lattice sequence in $\V$ of
$\Oo_\F$-period $e$. Write $\A=\End_\D(\V)$ and denote by $\e^j$ the
idempotents in $\PP_0(\La)$ corresponding to the decomposition of
$\V$. We put $\b=\sum_{j=1}^r \b_j$. 
Then $[\La,n,n-1,\b]$ is a stratum in $\A$.

\begin{coro}\label{cor:minpol}
With notation as above, suppose the strata $[\La^j,n,n-1,\b_j]$ are all
equivalent to simple (or null) strata. Then they have the same
minimum polynomial if and only if 
$[\La,n,n-1,\b]$ is equivalent to a simple (or null) stratum.
\end{coro}

\begin{proof} Writing $\A=\bigoplus_{i,j}\End_\D(\V^j,\V^i)$, it is clear that
$\y(\b,\La)$ is block diagonal of the form
$\diag(\y(\b_1,\La^1),\ldots,\y(\b_r,\La^r))$. In particular, the minimum
polynomial of $\overline \y(\b,\La)$ is the $\gcd$ of the minimum polynomials
of $\overline \y(\b_j,\La^j)$. The result is now immediate from
Proposition~\ref{prop:simpleminpol}, with the case of null strata coming from the case where the minimum polynomial is $X$.
\end{proof}

\begin{coro}\label{cor:minpolpair}
Let $(k,\b)$ be a simple pair over $\F$ and let $[\La^1,n_1,m_1,\varphi_1(\b)]$ be
a realization in some simple central $\F$-algebra $\A^1$. 
The minimum polynomial of $[\La^1,n_1,m_1,\varphi_1(\b)]$ depends only on  the
endo-equivalence class of the pair $(-k_\F(\b)-1,\b)$.
\end{coro}

\begin{proof} First note that Corollary~\ref{cor:minpol} essentially
says that the minimum polynomial is independent of the realization.
For suppose $[\La^2,n_2,m_2,\varphi_2(\b)]$ is another realization in a simple
central $\F$-algebra $\A^2$. Since the minimum polynomial depends only
on the induced strata in $\widetilde\A^1$ and $\widetilde\A^2$, we may as
well suppose that both algebras are split -- that is,
$\A^j=\End_\F(\V^j)$, for some $\F$-vector space $\V^j$,
$j=1,2$. Moreover, by scaling we may assume that $\La^1$ and $\La^2$
have the same period so that $n_1=n_2=n$. Put $m=\max\{m_1,m_2\}$.

Now set $\V=\V^1\oplus\V^2$ and use the notation of
Corollary~\ref{cor:minpol}; also let $\varphi=\varphi_1+\varphi_2$ denote the
diagonal embedding of $\F[\b]$ in $\A$. The stratum
$[\La,n,m,\varphi(\b)]$ is then a realization of the simple pair $(k,\b)$
so it is simple and $[\La,n,n-1,\varphi(\b)]$, being pure, is equivalent
to a simple stratum. Hence, by Corollary~\ref{cor:minpol}, the
strata $[\La^j,n,n-1,\varphi_j(\b)]$ have the same minimum polynomial.

Finally, suppose $(k,\g)$ is a simple pair endo-equivalent to
$(k,\b)$. Let $[\La,n,m,\varphi(\b)]$ and $[\La,n,m,\rho(\g)]$ be
realizations in some split simple central $\F$-algebra
$\A=\End_\F(\V)$. By Proposition~\ref{prop:IICsimplepairs},
these strata are conjugate by some $u\in\KK(\La)$ so, by
conjugating the embedding $\rho$, we may assume that
$[\La,n,n-1,\varphi(\b)]$ is equivalent to $[\La,n,n-1,\rho(\g)]$. In
particular, they have the same minimum polynomial. 
\end{proof}

\subsection{}
The following is a generalization of~\cite[Lemma~2.4.11]{BK},~\cite[Lemma~1.9]{G}:

\begin{lemm}\label{lem:scalar}
Let $[\La,n,n-1,b]$ be a stratum in $\A$. It is intertwined by every
element of~$\G$ if and only if $(b+\PP_{1-n}(\La))\cap\F\ne\emptyset$.
\end{lemm}

\begin{proof} The proof follows the same scheme as that of~\cite[Lemma~2.4.11]{BK} (see \emph{op.\ cit.} pp.77--78) so we only sketch the argument. Suppose $[\La,n,n-1,b]$ is intertwined by every
element of~$\G$. If $b\in\PP_{1-n}(\La)$ then $0\in(b+\PP_{1-n}(\La))\cap\F$ so we assume $b\not\in\PP_{1-n}(\La)$. Then $b$ defines a non-zero map $\overline\b$ in
$$
\PP_{-n}(\La)/\PP_{1-n}(\La) = \bigoplus_{i=0}^{e-1}\Hom_{k_\D}(\La(i)/\La(i+1),\La(i-n)/\La(i-n+1)),
$$
where $e=e(\La|\Oo_\D)$ is the $\Oo_\D$-period of $\La$. (Note that there is, in general, redundancy in this sum: the spaces $\Hom_{k_\D}(\La(i)/\La(i+1),\La(i-n)/\La(i-n+1))$ may be $0$.) 

Since $\overline\b$ is non-zero, by moving $\La$ in its translation
class we can suppose that it defines a non-zero map in
$\Hom_{k_\D}(\La(0)/\La(1),\La(-n)/\La(1-n))$. If $e\nmid n$ then we
can find $g\in\G\cap\PP_0(\La)$ such that $g$ induces the identity map
on $\La(0)/\La(1)$ but the zero map on $\La(i)/\La(i+1)$, for $1\le
i\le e-1$. But then $\overline{gb}$ is zero on $\La(0)/\La(1)$, while
$\overline{bg}$ coincides with $\overline b$ on $\La(0)/\La(1)$, so is
non-zero, which contradicts the assumption that $g$ intertwines
$[\La,n,n-1,b]$.

Thus $e$ divides $n$ and we put $t=-n/e$. Fix $\L/\F$ a maximal
unramified subfield of $\D$, and $\w_\D$ a uniformizer of $\D$ which
normalizes $\L$ (so acts via conjugation as a generator of
$\Gal(\L/\F)$) and such that $\w_\D^d=\W_\F$.
Then the coset $b\w_\D^{-t}+\PP_1(\La)$ is intertwined by every
$g\in\G$ which commutes with $\w_\D$. In particular, since elementary
matrices (with respect to a suitable basis) commute with $\w_\D$, we
find that $(b\w_\D^{-t}+\PP_1)\cap\Oo_\D\ne\emptyset$ and, since
conjugation by $\w_\D$ acts by Frobenius on $k_\D$, the fact that
$\w_\D$ intertwines implies that
$(b\w_\D^{-t}+\PP_1)\cap\Oo_\F\ne\emptyset$. Thus $b\equiv
\w_\D^t\l\pmod{\PP_{1-n}(\La)}$, for some
$\l\in\Oo_\F^\times$. Finally, the fact that every element of
$\Oo_\L^\times$ intertwines the stratum implies that conjugation by
$\w_\D^t$ acts trivially on $k_\D$, so $d$ divides $t$ and
$\w_\D^t\l\in\F$, as required.

The converse is trivial.
\end{proof}

\subsection{}
Let $\A$ be a simple central $\F$-algebra and $\V$ be a simple left
$\A$-module. Let $[\La,n,m,\b]$ be a simple stratum, set
$\E=\E_\b=\F[\b]$, and let $\B=\B_\b$ denote the $\A$-centralizer of
$\E$. We identify $\widetilde\A=\End_\F(\V)$ with
$\A\otimes_\F\End_\A(\V)$. From~\cite[D\'efinition~2.25]{SS4} (see
also~\cite[\S4.2]{Br}) a \emph{tame corestriction relative to $\E/\F$}
on $\A$ is a $(\B,\B)$-bimodule homomorphism $s=s_\b:\A\to\B$ such
that $\widetilde s=s\otimes\id_{\End_\A(\V)}$ is a tame corestriction
relative to $\E/\F$ on $\widetilde\A$, in the sense
of~\cite[Definition~1.3.3]{BK}.

\begin{lemm}[{cf.~\cite[Lemma~2.4.12]{BK}}]
\label{lem:cruxstrata}
Let $[\La,n,m,\b_i]$ be equivalent simple strata. Then, putting
$\E_i=\F[\b_i]$, we
have:
\begin{enumerate}
\item $\Oo_{\E_1}+\PP_1(\La)=\Oo_{\E_2}+\PP_1(\La)$;
\item there are tame corestrictions $s_i$ on $\A$ relative to
  $\E_i/\F$ such that, for all $k\in\ZZ$ and $x\in\PP_k(\La)$,
$$
s_1(x) \equiv s_2(x) \pmod{\PP_{k+1}(\La)};
$$
\item there are prime elements $\varpi_i$ of $\E_i$ such that
  $\varpi_1\U^1(\La)=\varpi_2\U^1(\La)$;
\item the pairs $(\E_i,\La)$ have the same embedding type.
\end{enumerate}
\end{lemm}

\begin{proof} We begin by proving (i)-(iii), for which we may assume
that $\La$ is strict by passing first to $\La^\ddag$
(cf.~paragraph~\ref{SS:passtoddag}). To return to $\Lambda$, notice that the
condition in (iii) is equivalent to
$\w_1\equiv\w_2\pmod{\PP_{e+1}(\La)}$, with $e=e(\La|\Oo_{\E_i})$, so
for (i) and (iii) we can simply intersect with $\A$. The same is true
for (ii) since tame corestrictions also decompose by blocks,
by~\cite[Proposition~2.26]{SS4}.

The simple strata $[\widetilde\La,n,m,\b_i]$ in $\widetilde\A$ are
equivalent so, by~\cite[Lemma~2.4.12]{BK}, we have the results
corresponding to (i) and (ii) in $\widetilde\A$, while (iii) is a
by-product of the proof (see also~\cite[Lemma~5.2]{BKcc} and its
proof). Intersecting with $\A$ gives the result here.

In the case of strict sequences, (iv) is given
by~\cite[Lemma~5.2]{BG}. Moreover, writing $\Ll$ for the strict
lattice sequence with the same image as $\La$, since
$\PP_0(\La)=\PP_0(\Ll)$ and $\PP_1(\La)=\PP_1(\Ll)$, the same proof
(using \emph{op.\ cit.}~Lemma~2.3.6) works in the general case to show
that the maximal unramified subextensions of $\E_i$ are conjugate in
$\U^1(\Ll)=\U^1(\La)$.
\end{proof}

\subsection{}
Now let $\V_\E$ be a simple left $\B$-module, let $\D_\E$ be the
algebra opposite to $\End_\B(\V_\E)$, and let  $\Ga=\Ga_\b$ be the
unique (up to translation) $\Oo_{\D_\E}$-lattice sequence $\V_\E$ such
that $\PP_k(\La)\cap\B = \PB_k(\Ga)$, for all $k\in\ZZ$.

\begin{defi}[{\cite[D\'efinition~3.21]{SS4}}] A \emph{derived stratum} of $[\La,n,m,\b]$ is a stratum of the form $[\Ga,m,m-1,s(b)]$, for some $b\in\PP_{-m}(\La)$ and some tame corestriction $s$ on $\A$ relative to $\E/\F$.
\end{defi}

The following result is a slight strengthening of~\cite[Proposition~3.30]{SS4}: 

\begin{prop}
\label{prop:refinement}
Let $[\La,n,m,\b]$ be a simple stratum and let $b\in\PP_{-m}(\La)$ be
such that the derived stratum $[\Ga,m,m-1,s(b)]$ is equivalent to a simple (or null) stratum $[\Ga,m,m-1,c]$. Then there is a simple stratum $[\La,n,m-1,\b']$ equivalent to $[\La,n,m-1,\b+b]$ and, moreover, for any such stratum, writing $\E'=\F[\b']$ and $\E_1=\F[\b,c]$, we have:
\begin{enumerate}
\item $e(\E'/\F)=e(\E_1/\F)$, $f(\E'/\F)=f(\E_1/\F)$ and
$k_0(\b',\La)=\begin{cases} k_0(\b,\La) &\hbox{ if }c\in\E, \\
-m &\hbox{ otherwise;} \end{cases}$

\medskip
\item $
\Oo_{\E'}+\PP_1(\La) = \Oo_{\E_1}+\PP_1(\La) = \Oo_{\E}[\y(s(b))] + \PP_1(\La).
$
\end{enumerate}
\end{prop}

\begin{proof} The first assertion is proved in~\cite[Proposition~3.30]{SS4}, under the hypothesis that the derived stratum $[\Ga,m,m-1,s(b)]$ is itself
simple. If it is only equivalent to the simple stratum
$[\Ga,m,m-1,c]$ then $c-s(b)\in\PB_{1-m}(\Ga)$ so,
by~\cite[Proposition~2.29]{SS4}, there is $d\in\PP_{1-m}(\La)$ such
that $s(d)=c-s(b)$. Replacing $b$ by $b+d$ we reduce to the case that
the derived stratum is simple and the result here follows.

For (i), we may pass first to $\La^\ddag$, where the result follows from~\cite[Proposition~9.5]{G}.

For (ii), the second equality follows from Proposition~\ref{prop:simpleminpol}, while the independence of $\Oo_{\E'}+\PP_1(\La)$ on the choice of $\beta'$ comes from Lemma~\ref{lem:cruxstrata}. In particular, we need only find a single $\b'$ for which the first equality holds. 

We fix a simple right $\E_1\otimes_\F\D$-module $\rS$ and put
$\A(\rS)=\End_\D(\rS)$. Let $\C$ be the commutant of $\E_1$ in $\A$,
and let $\D_1$ be the commutant of $\E_1$ in $\A(\rS)$. We also fix a
decomposition $\V=\V^1\oplus\cdots\oplus\V^l$ into simple right
$\E_1\otimes_\F\D$-modules (which are all copies of $\rS$) such that
the lattice sequence $\La$ decomposes into the direct sum of the
$\La^j=\La\cap\V^j$, for $j\in\{1,\ldots,l\}$. From~\cite[\S1.3]{VS3},
after choosing identifications $\V^i\simeq \rS$, we have an $\F$-algebra
embedding $\ii:\A(\rS)\to\A$ and an isomorphism of
$(\A(\rS),\C)$-bimodules
$$
\A(\rS)\otimes_{\D_1}\C \to \A.
$$
We denote by $\B(\rS)$ the commutant of $\E$ in $\A(\rS)$ and let $\rS_\E$ be a simple left $\B(\rS)$-module. By~\cite[Lemme~3.31]{SS4}, the tame corestriction $s$ on $\A$ relative to $\E/\F$ takes the form $s_1\otimes\id_\C$, for $s_1$ a tame corestriction on $\A(\rS)$ relative to $\E/\F$.

Denote by $\SS$ the unique (up to translation) $\E_1$-pure strict $\Oo_\D$-lattice sequence on $\rS$, and by $\SS_\b$ the unique (up to translation) $\Oo_{\D_\E}$-lattice sequence on $\rS_\E$ compatible with the filtration from $\SS$. Set $n_0=-v_{\SS}(\b)$ and $m_0=-v_{\SS}(c)$ and pick $b_0\in\PP_{-m_0}(\SS)$ such that $s_1(b_0)=c$. By~\cite[Lemme~3.32]{SS4}, the stratum $[\SS,n_0,m_0-1,\b+b_0]$ is pure, so equivalent to a simple stratum $[\SS,n_0,m_0-1,\b+b']$, with $s_1(b')\in c+\PB_{1-m_0}(\SS_\b)$. We put $\E'=\F[\b+b']$.

Suppose first that $c\in\E$. Then, by (i), we have $k_0(\b+b',\SS)=k_0(\b,\SS)$ so that $[\SS,n_0,m_0,\b]$ and $[\SS,n_0,m_0,\b+b']$ are equivalent simple strata. Now Lemma~\ref{lem:cruxstrata}(i) implies
\begin{equation}\label{eqn:oEinAScase1}
\Oo_{\E'}+\PP_1(\SS) = \Oo_{\E_1}+\PP_1(\SS).
\end{equation}\addtocounter{equation}{-1}

Now suppose $c\not\in\E$. Let $x\in\Oo_{\E'}$ and put $r=-k_0(\b,\SS)$, which is strictly greater than $m_0$. Then certainly $a_\b(x)\in\PP_{-m_0}(\SS)$ so, by~\cite[Proposition~2.29]{SS4}, we can write $x=\g+y$, with $\g\in\PB_0(\SS)$ and $y\in\PP_{r-m_0}(\SS)\subseteq\PP_1(\SS)$. Thus
$$
0\ =\ (\b+b')(\g+y)-(\g+y)(\b+b')\ \equiv\ a_\b(y)+b'\g-\g b' \pmod{\PP_{1-m_0}(\SS)}.
$$
Applying $s_1$ and using $s_1(b')\in c+\PB_{1-m_0}(\SS)$, we see that $a_{c}(\g)\in\PB_{1-m_0}(\SS_\b)$.
Now $k_0(c,\SS_\b)=-m_0$ so we deduce that $\g\in\Oo_{\D_1}+\PB_1(\SS_\b)$, since $\D_1$ is the commutant of $\E_1=\E[c]$ in $\B(\rS)$. 
In particular, we see that $\Oo_{\E'}\subseteq\Oo_{\D_1}+\PP_1(\SS)$ so that the image of the residue field $k_{\E'}$ in $\PP_0(\SS)/\PP_1(\SS)$ is contained in the image of $k_{\D_1}$. Since (the images of) $k_{\E'}$ and $k_{\E_1}$ are then subfields of $k_{\D_1}$ of the same cardinality (by (i)), they must coincide and we deduce again that 
\begin{equation}\label{eqn:oEinAS}
\Oo_{\E'}+\PP_1(\SS) = \Oo_{\E_1}+\PP_1(\SS).
\end{equation}

Finally, we must translate this back to $\A$, using the embedding $\ii:\A(\rS)\to\A$; we will identify $\b$ and $c$ with their images under $\ii$. The image of the simple stratum $[\SS,n_0,m_0-1,\b+b']$  under $\ii$ is a simple stratum $[\La,n,m-1,\b+\ii(b')]$, and we have
$$
s(\ii(b')) = \ii(s_1(b'))\ \equiv c \equiv s(b) \pmod{\PB_{1-m}(\Ga)},
$$
since $\ii(\PP_{1-m_0}(\SS))\subseteq\PP_{1-m}(\La)$. Thus, as in the end of the proof of~\cite[Proposition~3.30]{SS4}, there exists $h\in\U_1(\La)$ such that, putting $\b'=h^{-1}(\b+\ii(b'))h$, we get a simple stratum $[\La,n,m-1,\b']$ equivalent to $[\La,n,m-1,\b+b]$. Then $\F[\b']=h^{-1}\ii(\E')h$ and
$$
\Oo_{\F[\b']} = h^{-1}\ii(\Oo_{\E'})h \equiv \ii(\Oo_{\E'}) \pmod{\PP_1(\La)}.
$$
Finally, by~\eqref{eqn:oEinAS}, we have $\ii(\Oo_{\E'})\equiv \ii(\Oo_{\E_1}) \pmod{\PP_1(\La)}$ and, since we have identified $\E_1$ with $\ii(\E_1)$, we deduce
$$
\Oo_{\F[\b']} + \PP_1(\La) = \Oo_{\E_1} + \PP_1(\La).
$$
This completes the proof of Proposition~\ref{prop:refinement}.
\end{proof}

\subsection{}We also have a converse to Proposition~\ref{prop:refinement}:

\begin{prop}[{cf.~\cite[Theorem~2.4.1]{BK},~\cite[Proposition~9.3]{G}}]
\label{prop:derived}
Let $[\La,n,m,\b]$ be a pure stratum equivalent to the simple stratum
$[\La,n,m,\g_1]$. Let $s_1$ be a tame corestriction on $\A$ relative to
$\F[\g_1]/\F$. Then the derived stratum $[\Ga_1,m,m-1,s_1(\b-\g_1)]$ is
equivalent to a simple (or null) stratum.
\end{prop}

We will need the following Lemma, which is in fact a special case of the Proposition.

\begin{lemm}[{cf.~\cite[(2.4.10)]{BK}}]\label{lem:simplederived}
Let $[\La,n,m,\b_i]$ be equivalent simple strata and let $s_1$ be a
tame corestriction on $\A$ relative to $\F[\b_1]/\F$. Then there
exists $\d\in\F[\b_1]$ such that $s_1(\b_1-\b_2)\equiv
\d\pmod{\PP_{1-m}(\La)}$.
\end{lemm}

\begin{proof}By passing to $\La^\ddag$, we may assume we are in the
strict sound case. The proof is then identical to that
of~\cite[(2.4.10)]{BK}, replacing~\cite[Proposition~1.4.6]{BK}
by~\cite[Proposition~4.3.3]{Br} and~\cite[Theorem~1.5.8]{BK}
by~\cite[Proposition~4.1.1]{Br}.
\end{proof}

\begin{proof}[Proof of Proposition~\ref{prop:derived}] By passing to an
equivalent stratum, we may assume that the stratum $[\La,n,m-1,\b]$
is simple. If $[\La,n,m,\b]$ is also simple then the result follows from
Lemma~\ref{lem:simplederived}; thus we may assume $k_0(\b,\La)=-m$. By
passing to  
$\La^\ddag$, we may assume we are in the strict sound case. We write
  $\E_1=\F[\g_1]$ and $\B_1$ for the $\A$-centralizer of $\E_1$.

By~\cite[Propositions~3.8,~9.3]{G}, there exists a simple stratum
$[\La,n,m,\g_2]$ equivalent to $[\La,n,m,\b]$ such that the derived
stratum $[\Ga_2,m,m-1,s_2(\b-\g_2)]$ is equivalent to a simple
(or null) stratum, where $s_2$ is a tame corestriction on $\A$
relative to $\F[\g_2]/\F$. Moreover, this derived stratum is
non-scalar, by Proposition~\ref{prop:refinement}, since $k_0(\b,\La)
> k_0(\g_2,\La)$, and thus, by Proposition~\ref{prop:simpleminpol}, it
has irreducible minimum polynomial.

We write $\E_2=\F[\g_2]$ and $\B_2$ for the
$\A$-centralizer of $\E_2$. By Lemma~\ref{lem:cruxstrata}, we may
assume that the tame corestriction $s_2$ is chosen such that, for
all $k\in\ZZ$ and $x\in\PP_k(\La)$,
$$
s_1(x) \equiv s_2(x) \pmod{\PP_{k+1}(\La)}.
$$ 
We also use Lemma~\ref{lem:cruxstrata} to choose uniformizers
$\w_i$ for $\E_i$ such that $\w_1\U^1(\La)=\w_2\U^1(\La)$.
Again by Lemma~\ref{lem:cruxstrata}, the residue fields $k_{\E_i}$ have
a common image in $\PP_0(\La)/\PP_1(\La)$ so that we may identify
them. Moreover, $\PB_0(\Ga_i)/\PB_1(\Ga_i)$ have a common image in
$\PP_0(\La)/\PP_1(\La)$: by~\cite[Proposition~2.29]{SS4}, the maps
$s_i:\PP_k(\La)\to\PP_k(\Ga_i)$ are surjective, for all $k\in\ZZ$, so
$\PB_0(\Ga_i)/\PB_1(\Ga_i)$ are the common image of $s_i$ in
$\PP_0(\La)/\PP_1(\La)$.

Put $b_i=\b-\g_i$. By the choices of $\w_i$ and $s_i$, we have
$$
\y(s_1(b_2),\Ga_1) \equiv \y(s_2(b_2),\Ga_2) \pmod{\PP_1(\La)}.
$$
In particular (given the identifications we have made), we see that
the two strata $[\Ga_i,m,m-1,s_i(b_2)]$ have the same minimum and
characteristic polynomials. In particular, $[\Ga_1,m,m-1,s_1(b_2)]$ has
irreducible minimum polynomial so, by
Proposition~\ref{prop:simpleminpol}, it is equivalent to a simple
stratum $[\Ga_1,m,m-1,c]$.

Finally, by Lemma~\ref{lem:simplederived}, the stratum
$[\Ga_1,m,m-1,s_1(\g_2-\g_1)]$ is equivalent to a scalar stratum
$[\Ga_1,m,m-1,\d]$, whence $[\Ga_1,m,m-1,s_1(b_1)]$ is equivalent
to the simple stratum $[\Ga_1,m,m-1,c+\d]$, as required.
\end{proof}


\section{Simple characters and refinement}\label{S.characters}

\subsection{} Let $\A$ be a central simple $\F$-algebra and let $[\La,n,0,\b]$ be a simple stratum in $\A$. 
To this simple stratum, in~\cite[\S2.4]{SS4}, one attaches 
compact open subgroups $\H(\b,\La)\subseteq\J(\b,\La)$ of $\mult\A$, together with filtrations
$$
\H^{m+1}(\b,\La)=\H(\b,\La)\cap\U^{m+1}(\La),\quad
\J^{m+1}(\b,\La)=\J(\b,\La)\cap\U^{m+1}(\La),\qquad m\ge 0,
$$ 
and a finite set $\Cc(\La,0,\b)$ of characters of $\H^{1}(\b,\La)$,
called \emph{simple characters of level $0$}, depending on the choice of
an additive character
$$
\psi_\F:\F\to\mathbb{C}^{\times}
$$
which is trivial on $\p_{\F}$ but not on $\Oo_{\F}$, and which will now
be fixed once and for all.

By restriction to $\H^{m+1}(\b,\La)$, we get also a set $\Cc(\La,m,\b)$ of \emph{simple characters of level $m$}. If $\lfloor n/2\rfloor\<m$, then $\H^{m+1}(\b,\La)=\U^{m+1}(\La)$, 
and the set $\Cc(\La,m,\b)$ reduces to the single character 
$\psi_\b$ of $\U^{m+1}(\La)$ defined by
$$
\psi_\b:x\mapsto\psi_\F\circ\tr_{\A/\F}(\b(x-1)),
$$
where $\tr_{\A/\F}$ denotes the reduced trace of $\A$ over $\F$, which
depends only on the equivalence class of $[\La,n,m,\b]$. 
More generally, for any $m$, the subgroup
$\H^{m+1}(\b,\La)$ and the set $\Cc(\La,m,\b)$ depend only on the
equivalence class of $[\La,n,m,\b]$. 

Note that we will use the following common convention: the trivial
character of the group $\U^{t+1}(\La)$ will be considered as a simple
character for the trivial stratum $[\La,t,t,0]$.

\subsection{}
Various properties of simple characters can be found in~\cite{SS4,BSS}. For now we recall two of them, the first of which is a special case of the intertwining formula~\cite[Th\'eor\`eme~2.24]{SS4}:

\begin{prop} 
Let $\A$ be a central simple $\F$-algebra, let $[\La,n,0,\b]$ be a simple stratum in $\A$ and let $\t\in\Cc(\La,0,\b)$. Then, writing $\B$ for the $\A$-centralizer of $\b$ as usual, we have 
$$
\I_\G(\t)=\J^1(\b,\La)\mult\B\J^1(\b,\La).
$$
\end{prop}

The following fundamental result is one of the main results of~\cite{BSS}:

\begin{prop}[{\cite[Theorem~1.12]{BSS}}]
\label{prop:IICwithK}
Let $\A$ be a simple central $\F$-algebra.
For $i=1,2$, let $[\La,n,m,\b_i]$ be a simple stratum in $\A$, 
and let $\t_{i}\in\Cc(\La,m,\b_i)$ be a simple character.
Write $\K_i$ for the maximal un\-ramified extension of $\F$ contained
in $\F(\b_i)$.
Assume that $\t_1$ and $\t_2$ inter\-twine in $\mult\A$ and that the $(\F[\b_i],\La)$ have the same embedding type.
Then there is an element $u\in\KK(\La)$ such that: 
\begin{enumerate}
\item 
$\K_1=u\K_2u^{-1}$; 
\item
$\Cc(\La,m,\b_1)=\Cc(\La,m,u\b_2u^{-1})$;
\item 
$\t_2(x)=\t_1(uxu^{-1})$, for all $x\in\H^{m+1}(\b_2,\La)$.
\end{enumerate}
\end{prop}

\subsection{} 
One of the technical difficulties with simple characters is that they
do not determine the simple stratum used to define them: that is, we
may have $\Cc(\La,m,\b_1)\cap\Cc(\La,m,\b_2)\ne\emptyset$, for
inequivalent strata $[\La,n,m,\b_i]$ (though certain invariants of the
strata are equal -- see later). In order to cope with this, we need
the following translation principle, which is the main result of this
section.

\begin{theo}[{cf.~\cite[Translation~Principle~2.11]{BKcc}}]
\label{thm:translation}
Let $[\La,n,m,\g_i]$ be simple strata with
$\Cc(\La,m,\g_1)\cap\Cc(\La,m,\g_2)\ne\emptyset$. Let $[\La,n,m-1,\b_1]$ be a simple stratum such that $[\La,n,m,\b_1]$ is equivalent to $[\La,n,m,\g_1]$. Then there is a simple stratum $[\La,n,m-1,\b_2]$
such that $[\La,n,m,\b_2]$ is equivalent to $[\La,n,m,\g_2]$ and
$\Cc(\La,m-1,\b_1)=\Cc(\La,m-1,\b_2)$. 
\end{theo}

The proof of this translation principle, which will take up most of
the remainder of this section, begins with the following
special case, in which $\b_1=\g_1$:

\begin{lemm}[{cf.~\cite[Theorem~3.5.9]{BK},~\cite[Proposition~9.10]{G}}]
\label{lem:translation}
Let $[\La,n,m,\g_i]$ be simple strata with
$\Cc(\La,m,\g_1)\cap\Cc(\La,m,\g_2)\ne\emptyset$. Then
$\H^m(\g_1,\La)=\H^m(\g_2,\La)$ and there is a simple stratum
$[\La,n,m,\b_2]$ equivalent to $[\La,n,m,\g_2]$ such that
$\Cc(\La,m-1,\b_2)=\Cc(\La,m-1,\g_1)$. 
\end{lemm}

\begin{proof} From~\cite[Lemma~4.12]{BSS}, we have
already that $\H^m(\g_1,\La)=\H^m(\g_2,\La)$. The remainder of the
proof is \emph{mutatis mutandis} that of~\cite[Theorem~3.5.9]{BK}: we
replace~\cite[Theorem~3.3.2]{BK} by~\cite[Th\'eor\`eme~2.23]{SS4},
\cite[Lemma~2.4.11]{BK} by Lemma~\ref{lem:scalar},
\cite[Theorem~2.2.8]{BK} by Proposition~\ref{prop:refinement}, 
\cite[3.3.20]{BK} by~\cite[Proposition~2.15]{SS4},
and~\cite[3.5.8]{BK} by~\cite[Theorem~4.16]{BSS}; 
for the proof of~\cite[Lemma~3.5.10]{BK} we
replace~\cite[Corollary~3.3.17]{BK} by~\cite[Proposition~2.24]{SS4}
and~\cite[Proposition~3.3.9]{BK} by~\cite[Lemme~2.30(2)]{SS4}.
\end{proof}

\subsection{}
The technical crux of the translation principle is contained in the
following lemma:

\begin{lemm}[{cf.~\cite[Lemma~5.2]{BKcc}}]
\label{lem:crux}
Let $[\La,n,m-1,\b_i]$ be simple strata with
$\Cc(\La,m-1,\b_1)\cap\Cc(\La,m-1,\b_2)\ne\emptyset$. Then, putting $\E_i=\F[\b_i]$, we
have:
\begin{enumerate}
\item $\Oo_{\E_1}+\PP_1(\La)=\Oo_{\E_2}+\PP_1(\La)$;
\item the pairs $(\E_i,\La)$ have the same embedding type;
\item there are prime elements $\varpi_i$ of $\E_i$ such that $\varpi_1\U^1(\La)=\varpi_2\U^1(\La)$;
\item there are tame corestrictions $s_i$ on $\A$ relative to
  $\E_i/\F$ such that, for all $k\in\ZZ$ and $x\in\PP_k(\La)$,
$$
s_1(x) \equiv s_2(x) \pmod{\PP_{k+1}(\La)}.
$$
\end{enumerate}
\end{lemm}

Note that (ii) in this lemma answers Conjecture~4.17 of~\cite{BSS} --
indeed, it is a generalization of that conjecture since here we do not
assume that the sequence $\La$ is strict. Also, the hypothesis
$\Cc(\La,m-1,\b_1)\cap\Cc(\La,m-1,\b_2)\ne\emptyset$ is equivalent to
$\Cc(\La,m-1,\b_1)=\Cc(\La,m-1,\b_2)$, by~\cite[Theorem~4.16]{BSS}.

\begin{proof} In the split case when $\La$ is strict, this
is~\cite[Lemma~5.2]{BKcc}, while the case of arbitrary $\La$ follows
by passing to $\La^\ddag$.

We proceed by induction on $m$. When $m=n$ the result is
immediate from Lemma~\ref{lem:cruxstrata}. Note that $k_0(\b_i,\La)$
is independent of $i$, by~\cite[Lemma~4.7]{BSS}. If 
$k_0(\b_i,\La)<-m$ then again the result is clear from the induction
hypothesis, since the conclusions (i)--(iv) are independent of $m$, so 
we assume $k_0(\b_i,\La)=-m>-n$. 
By replacing $[\La,n,m-1,\b_1]$ by an equivalent
stratum, Lemma~\ref{lem:translation} says we may (and do) assume
$\Cc(\La,0,\b_1)=\Cc(\La,0,\b_2)$ without affecting the conclusion of
the lemma, by Lemma~\ref{lem:cruxstrata}.

For $i=1,2$, let $[\La,n,m,\g_i]$ be a simple stratum equivalent to
$[\La,n,m,\b_i]$. Then
$$
\Cc(\La,m,\g_1)=\Cc(\La,m,\b_1)=\Cc(\La,m,\b_2)=\Cc(\La,m,\g_2)
$$
so, by induction applied to the simple strata $[\La,n,(m+1)-1,\g_i]$,
(i)--(iv) are satisfied with $\E_{\g_i}=\F[\g_i]$ in place of
$\E_i$, and we pick uniformizers $\w_{\g_i}$ and tame corestrictions
$s_{\g_i}$ satisfying (iii),~(iv). Moreover, by replacing
$[\La,n,m,\g_1]$ by an equivalent stratum, Lemma~\ref{lem:translation}
says we may assume 
$\Cc(\La,0,\g_1)=\Cc(\La,0,\g_2)$.

Put $c_i=\b_i-\g_i$ and let
$\t\in\Cc(\La,m-1,\b_1)=\Cc(\La,m-1,\b_2)$. By~\cite[Proposition~2.15]{SS4},
we have $\t=\vartheta_i\psi_{c_i}$, for some
$\vartheta_i\in\Cc(\La,m-1,\g_i)$. Hence 
$$
\vartheta_1 = \vartheta_2\psi_{c_2-c_1}
$$
and $\vartheta_1,\vartheta_2\in\Cc(\La,m-1,\g_1)$ both restrict to the
same character $\vartheta\in\Cc(\La,m,\g_1)$. Since
$\vartheta_1,\vartheta_2$ are both intertwined by $\B_{\g_1}^\times$,
the same is true of $\psi_{c_2-c_1}$. In particular, restricting to
$\H^m(\g_1,\La)\cap\B_{\g_1}^\times=\U^m(\Ga_{\g_1})$, we see that
$\psi_{s_{\g_1}(c_2-c_1)}|\U^m(\Ga_{\g_1})$ is intertwined by all of
$\B_{\g_1}^\times$ and, by Lemma~\ref{lem:scalar},
$(s_{\g_1}(c_2-c_1)+\PB_{1-m}(\Ga_{\g_1}))\cap
\E_{\g_1}\ne\emptyset$.
In particular, the stratum $[\Ga_{\g_1},m,m-1,s_{\g_1}(c_2-c_1)]$ is
equivalent to a simple (or null) scalar stratum.

By Proposition~\ref{prop:refinement}, there is a simple stratum
$[\La,n,m-1,\g_1']$ equivalent to the stratum
$[\La,n,m-1,\g_1+(c_2-c_1)]$.
Since $\vartheta_2\in\Cc(\La,m-1,\g_1)$,
by~\cite[Proposition~2.15]{SS4} we have
$\vartheta_1=\vartheta_2\psi_{c_2-c_1}=\vartheta_2\psi_{\g'_1-\g_1}\in\Cc(\La,m-1,\g'_1)$. Hence
$\Cc(\La,m-1,\g'_1)=\Cc(\La,m-1,\g_1)$. Moreover, putting
$c_1'=\b_1-\g_1'$ we see that $c_2-c'_1\in\PP_{1-m}(\La)$; in
particular, for any tame corestriction $s_{\g'_1}$ on $\A$ relative to
$\E_{\g_1'}/\F$, we have $s_{\g_1'}(c_2-c_1')\in\PP_{1-m}(\La)$.

Thus, replacing $\g_1$ by $\g_1'$ we may assume that $s_{\g_1}(c_2-c_1)\in\PP_{1-m}(\La)$. By (iv), we also have $s_{\g_2}(c_2-c_1)\in\PP_{1-m}(\La)$. In particular, looking at the derived strata $[\Ga_{\g_i},m,m-1,s_{\g_i}(c_j)]$, with $i,j\in\{1,2\}$, and using the inductive hypotheses (iii),~(iv), we get
$$
\y(s_{\g_1}(c_1)) \equiv 
\y(s_{\g_1}(c_2)) \equiv 
\y(s_{\g_2}(c_2)) \pmod{\PP_1(\La)}.
$$
(The elements $\y$ here are computed with respect to the uniformizers $\w_{\g_i}$ satisfying (iii).)
By Proposition~\ref{prop:derived} the derived strata
$[\Ga_{\g_i},m,m-1,s_{\g_i}(c_i)]$ are equivalent to simple or null
strata so, by Proposition~\ref{prop:refinement} (applied to the strata
$[\La,n,m,\g_i]$ and $\b'=\b_i$) and the inductive hypothesis (i), we have
$$
\Oo_{\E_1}+\PP_1(\La) = 
\Oo_{\E_{\g_1}}[\y(s_{\g_1}(c_1))]+\PP_1(\La) = 
\Oo_{\E_{\g_2}}[\y(s_{\g_2}(c_2))]+\PP_1(\La) = 
\Oo_{\E_2}+\PP_1(\La) 
$$
and we have proved (i). Now (ii) follows exactly as in the proof of Lemma~\ref{lem:cruxstrata} (see also~\cite[Lemma~5.2]{BG}); indeed the proof gives the existence of $u\in\U^1(\La)$ such that $u^{-1}\K_1 u=\K_2$, where $\K_i$ is the maximal unramified subextension of $\E_i/\F$.

For the remainder, we may pass to $\La^\ddag$ and assume we have
soundly embedded strict lattice sequences with $\PP_0(\La)$ principal,
as in the proof of
Lemma~\ref{lem:cruxstrata}. (By~\cite[Proposition~4.11]{BSS}, we have
$\Cc(\La^\ddag,m-1,\b_1)=\Cc(\La^\ddag,m-1,\b_2)$; cf.\ the proof
of~\emph{op. cit.} Theorem~4.16.)
Recall that we have $\t\in\Cc(\La,m-1,\b_1)=\Cc(\La,m-1,\b_2)$, which
we extend to a common simple character in
$\Cc(\La,0,\b_1)=\Cc(\La,0,\b_2)$. Then $\t^u\in\Cc(\La,0,u^{-1}\b_1
u)$ surely intertwines $\t\in\Cc(\La,0,\b_2)$ so, by
Proposition~\ref{prop:IICwithK}, there is $g\in\KK(\La)$ such that
$\t^{ug}=\t$ and $(ug)^{-1}\K_1 (ug)=\K_2$. Since $ug$ then normalizes
$\t$, we have $ug\in\KK(\Ga_{\b_2})\J^1(\b_b,\La)$.   

In particular, there is $x\in\KK(\Ga_{\b_2})$ such that
$h=ugx\in\J^1(\b_2,\La)$, $h^{-1}\K_1 h=\K_2$ and $\t^h=\t$. Thus,
replacing $\b_1$ by $h^{-1}\b_1 h$, we may assume that $\K_1=\K_2=\K$,
without affecting the conclusion of the Lemma (since $h\in
\U^1(\La)$).

Now we will utilise the \emph{interior lifting} and \emph{base change} processes of~\cite{BSS} to reduce to the split case. 

We suppose first that we are in the special case $\K=\F$, that is
$\E_i/\F$ is totally ramified. Fix an unramified extension $\L/\F$
which splits $\A$, so that $\L_i=\E_i\otimes_\F\L$ is a field, for
$i=1,2$. The algebra $\overline\A=\A\otimes_\F\L$ is then a split
simple $\L$-central algebra and we choose a simple left
$\overline\A$-module $\overline\V$. There is a unique (up to
translation) strict $\Oo_\L$-lattice sequence $\overline\La$ on
$\overline\A$ such that
$\PP_k(\overline\La)=\PP_k(\La)\otimes_{\Oo_\F}\Oo_\L$, for all
$k\in\ZZ$ (see~\cite[\S2.2]{VS1}). Identifying $\A$ with $\A\otimes_\F
1 \subseteq \overline\A$, we get strata $[\overline\La,n,m-1,\b_i]$,
which are simple.

Denote by $\Cc(\overline\La,m-1,\b_i)$ the set of simple characters
with respect to the character $\psi_\F\circ\tr_{\L/\F}$. The base
change process from~\cite[\S7.2]{BSS} gives rise to injective
$\KK(\La)$-equivariant maps
$$
\bb^i = \bb^i_{\L/\F}:\Cc(\La,m-1,\b_i) \to \Cc(\overline\La,m-1,\b_i).
$$
Moreover, by~\cite[Proposition~7.6]{BSS}, we have
$\bb^1(\t)=\bb^2(\t)$ so
$\Cc(\overline\La,m-1,\b_1)\cap\Cc(\overline\La,m-1,\b_2)\ne\emptyset$. In
particular, by the split case we get uniformizers $\w^\L_i$ of $\L_i$
such that 
\begin{equation}
\label{eqn:wL}
\w^\L_1 + \PP_{e+1}(\overline\La) = \w^\L_2 + \PP_{e+1}(\overline\La),
\end{equation}
with $e=e(\overline\La|\Oo_\L)=e(\La|\Oo_\F)$,
and tame corestrictions $s_i^\L$ on $\overline\A$ relative to
$\L_i/\L$ for the character $\psi_\F\circ\tr_{\L/\F}$ such that, for
all $k\in\ZZ$ and $x\in\PP_k(\overline\La)$,
$$
s_1^\L(x) \equiv s_2^\L(x) \pmod{\PP_{k+1}(\overline\La)}.
$$
Multiplying through~\eqref{eqn:wL} by a unit, we see that we may
assume $\w^\L_1=\w_1$ is a uniformizer of $\E_1$ and $\w^\L_2=\w_2\z$,
for some uniformizer $\w_2$ of $\E_2$ and $\z\in\Oo_\L^\times$ a root
of unity of order coprime to $p$. Thus we have
$$
\w_1\w_2^{-1}\ \equiv\ \z \pmod{\PP_1(\overline\La)}.
$$
Now the Galois group $\Gal(\L/\F)$ acts on $\overline A$, fixing
$\w_1\w_2^{-1}$, so the image of $\z$ in $k_\L$ is fixed by
$\Gal(k_\L/k_\F)$. In particular, $\z\in\Oo_\F^\times$ so, replacing
$\w_2$ by $\w_2\z$, we get
$$
\w_1 + \PP_{e+1}(\overline\La) = \w_2 + \PP_{e+1}(\overline\La),
$$
and intersecting with $\A$ completes the proof of (iii).

The argument for the tame corestrictions is similar: We check that, if
$s_i$ is an arbitrary tame corestriction on $\A$ relative to
$\E_i/\F$, then $s_i\otimes 1$ is a tame corestriction on
$\overline\A$ relative to $\L_i/\L$. By~\cite[Proposition~2.26]{SS4},
$s_i^\L$ and $s_i\otimes 1$ differ by a unit $u_i$ in $\Oo_{\L_i}$
and, changing by a root of unity, we can assume $u_1\in 1+\p_{\L_1}$. 
We have $u_2\equiv\z\pmod{\p_{\L_i}}$, for some root of unity
$\z\in\Oo_\L^\times$. Then, for all $k\in\ZZ$ and $a\in\PP_k(\La)$
$$
s_1(a)\otimes 1\equiv s_1^\L(a\otimes 1) \equiv s_2^\L(a\otimes 1) \equiv s_2(a)\otimes \z \pmod{\PP_{k+1}(\overline\La)}.
$$
Again, the Galois group $\Gal(\L/\F)$ acts on $\overline A$, fixing
$s_1(a)\otimes 1$, so 
$$
s_2(a)\otimes \z \equiv s_2(a)\otimes \z^\s
\pmod{\PP_{1}(\overline\La)}, \hbox{\qquad for all
}a\in\PP_0(\La),\ \s\in\Gal(\L/\F).
$$
By~\cite[Proposition~2.29]{SS4}, the map
$s_2:\PP_0(\overline\La)\to\PP_0(\overline\Ga_2)$ is surjective so,
as above, we deduce that $\z\in\Oo_\F^\times$ and, after replacing
$s_2$ by $\z s_2$, we may assume $\z=1$. Finally, intersecting with
$\A$ completes with proof of (iv).

\medskip

Finally we consider the case where $\K\ne\F$. Denote by $\C$ the $\A$-centralizer of $\K$, fix a simple left $\C$-module $\W$, and let $\D_\K$ be the algebra opposite to $\End_\C(\W)$. Let $\Ga_\K$ be the unique (up to translation) $\Oo_{\D_\K}$-lattice sequence on $\W$ such that
$$
\PP_k(\La)\cap\C = \PP_k(\Ga_\K), \qquad k\in\ZZ.
$$
Then $[\Ga_\K,n,m-1,\b_i]$ is a simple stratum in $\C$, for $i=1,2$, by~\cite[Proposition~5.2]{BSS}. From~\cite[Theorem~5.8,~Proposition~6.12]{BSS}, we get \emph{interior lifting} maps
$$
\ll^i = \ll^i_{\K/\F}:\Cc(\La,m-1,\b_i) \to \Cc(\Ga_\K,m-1,\b_i),
$$
which are injective and $\KK(\Ga_\K)$-equivariant. Moreover, by~\cite[Proposition~6.13]{BSS}, we have $\ll^1(\t)=\ll^2(\t)$ so that $\Cc(\Ga_\K,m-1,\b_1)\cap\Cc(\Ga_\K,m-1,\b_2)\ne\emptyset$. From the case $\K=\F$ above, we find uniformizers $\w_i$ of $\E_i$ with $\w_1\U^1(\Ga_\K)=\w_2\U^1(\Ga_\K)$; in particular, $\w_1\U^1(\La)=\w_2\U^1(\La)$ which proves (iii). 
We also get tame corestrictions $s^\K_i$ on $\C$ relative to $\E_i/\K$ satisfying (iv): for all $k\in\ZZ$ and $x\in\PP_k(\Ga_\K)$,
$$
s_1^\K(x) \equiv s_2^\K(x) \pmod{\PP_{k+1}(\Ga_\K)}.
$$
Finally, if $s_\K$ is any tame corestriction on $\A$ relative to $\K/\F$ then $s_i=s_i^\K\circ s_\K$ are tame corestrictions on $\A$ relative to $\E_i/\F$, which satisfy (iv) since $\PP_{k+1}(\Ga_\K)\subseteq\PP_{k+1}(\La)$.
\end{proof}

\subsection{}
Now we are ready to complete the proof of the translation principle. 

\begin{proof}[Proof of Theorem~\ref{thm:translation}]
For $i=1,2$, we have simple strata $[\La,n,m,\g_i]$ such that
$\Cc(\La,m,\g_1)\cap\Cc(\La,m,\g_2)\ne\emptyset$; these sets of simple
characters are then equal, by~\cite[Theorem~4.16]{BSS}. Moreover, by
Lemma~\ref{lem:translation}, we may replace $[\La,n,m,\g_1]$ by an
equivalent stratum so that $\Cc(\La,m-1,\g_1)=\Cc(\La,m-1,\g_2)$.

Let $\B_i$ denote the $\A$-centralizer of $\F[\g_i]$, let $\V_i$ be a
simple left $\B_i$-module, and let $\D_i$ be the algebra opposite to
$\End_{\B_i}(\V_{\F[\g_i]})$. Denote by $\Ga_i$ the unique (up to
translation) $\Oo_{\D_i}$-lattice sequence in $\V_i$ such that 
$$
\PP_k(\La)\cap\B_i = \PB_k(\Ga_i),\quad k\in\ZZ.
$$

We use Lemma~\ref{lem:crux} to choose uniformizers $\w_i$ for $\F[\g_i]$ such that $\w_1\U^1(\La)=\w_2\U^1(\La)$, and tame corestrictions $s_i$ on $\A$ relative to $\F[\g_i]/\F$ such that, for all $k\in\ZZ$ and $x\in\PP_k(\La)$,
$$
s_1(x) \equiv s_2(x) \pmod{\PP_{k+1}(\La)}.
$$
Again by Lemma~\ref{lem:crux}, the residue fields $k_{\F[\g_i]}$ have
a common image in $\PP_0(\La)/\PP_1(\La)$ so that we may identify
them. Moreover, $\PB_0(\Ga_i)/\PB_1(\Ga_i)$ have a common image in
$\PP_0(\La)/\PP_1(\La)$, as in the proof of Proposition~\ref{prop:derived}.

\medskip

Now let $[\La,n,m-1,\b_1]$ be a simple stratum such that
$[\La,n,m,\b_1]$ is equivalent to $[\La,n,m,\g_1]$. If the stratum
$[\La,n,m,\b_1]$ is itself simple then the result follows from
Lemma~\ref{lem:translation} (applied with $\b_1$ in place of $\g_1$),
so we assume this is not the case. We write
$\b_1=\g_1+b$, with $b\in\PP_{-m}(\La)$, and pick a simple character
$\vartheta\in\Cc(\La,m-1,\g_1)$, so that $\vartheta\psi_b$ is a simple
character in $\Cc(\La,m-1,\b_1)$.

By Proposition~\ref{prop:derived}, the derived stratum
$[\Ga_1,m,m-1,s_1(b)]$ is equivalent to a simple stratum, which is
non-scalar by Proposition~\ref{prop:refinement}, since $k_0(\b_1,\La)
> k_0(\g_1,\La)$. Thus, by Proposition~\ref{prop:simpleminpol},
$[\Ga_1,m,m-1,s_1(b)]$ has irreducible minimum polynomial. However, by
the choices of $\w_i$ and $s_i$, we have
$$
\y(s_1(b),\Ga_1) \equiv \y(s_2(b),\Ga_2) \pmod{\PP_1(\La)}.
$$
In particular (given the identifications we have made), we see that
the strata $[\Ga_i,m,m-1,s_i(b)]$ have the same minimum and
characteristic polynomials. In particular, $[\Ga_2,m,m-1,s_2(b)]$ has
irreducible minimum polynomial so, by
Proposition~\ref{prop:simpleminpol}, it is equivalent to a simple
stratum.  

Finally, by Proposition~\ref{prop:refinement}, there is a simple
stratum $[\La,n,m-1,\b_2]$ equivalent to $[\La,n,m-1,\g_2+b]$ and,
by~\cite[Proposition~2.15]{SS4}, we have
$\vartheta\psi_b\in\Cc(\La,m-1,\b_2)$. In particular,
$\Cc(\La,m-1,\b_1)\cap\Cc(\La,m-1,\b_2)\ne\emptyset$ so,
by~\cite[Theorem~4.16]{BSS}, we have
$\Cc(\La,m-1,\b_1)=\Cc(\La,m-1,\b_2)$ as required.
\end{proof}

\subsection{}
We will need one corollary of the translation principle, which is in fact
a generalization of Proposition~\ref{prop:derived}:

\begin{coro}\label{cor:derived}
Let $[\La,n,m,\g]$ be a simple stratum and let
$[\La,n,0,\b]$ be a simple stratum such that
$\Cc(\La,m,\b)\cap\Cc(\La,m,\g)\ne\emptyset$. Suppose
$\vartheta\in\Cc(\La,m-1,\g)$ and $\theta\in\Cc(\La,m-1,\b)$ coincide on
$\H^{m+1}(\g,\La)=\H^{m+1}(\b,\La)$. Then there is $c\in\PP_{-m}(\La)$
such that $\theta=\vartheta\psi_c$ and, for any such $c$, the derived
stratum $[\Ga_\g,m,m-1,s_\g(c)]$ is equivalent to a simple (or null)
stratum. 
\end{coro}

\begin{proof} The entire statement depends only on the equivalence
class of the stratum $[\La,n,m-1,\b]$ so, by replacing with an
equivalent stratum, we may assume this stratum is simple.
Moreover, that such $c\in\PP_{-m}(\La)$ exists is clear so we need
only prove that the derived stratum is simple.

By the translation principle (Theorem~\ref{thm:translation}),
there is a simple stratum $[\La,n,m-1,\b']$ such that $[\La,n,m,\b']$
is equivalent to $[\La,n,m,\g]$ and 
$\Cc(\La,m-1,\b')=\Cc(\La,m-1,\b)$. 
Then, by~\cite[Proposition~2.15]{SS4}, we have
$\t|\H^m(\b',\La)=\vartheta'\psi_{c'}$, for some
$\vartheta'\in\Cc(\La,m-1,\g)$ and $c'=\b'-\g$. Moreover, the
derived stratum $[\Ga_\g,m,m-1,s_\g(c')]$ is simple (or null),
by Proposition~\ref{prop:derived}. 

Now we have two simple characters,
$\vartheta',\vartheta|\H^m(\g,\La)$ in $\Cc(\La,m-1,\g)$, which
differ by the character $\psi_{c-c'}$. Since the simple characters
are intertwined by $\B_\g^\times$, so is $\psi_{c-c'}$ and, in
particular, its restriction to $\U^m(\Ga_\g)$. Then,
by Lemma~\ref{lem:scalar}, there is a $\d\in \F[\g]$ such that
$s_\g(c-c')-\d\in\PB_{1-m}(\Ga)$. In particular,
$[\Ga_\g,m,m-1,s_\g(c)]$ is equivalent to
$[\Ga_\g,m,m-1,s_\g(c')+\d]$, which is simple (or null).
\end{proof}


\section{Endo-classes and common approximations}\label{S.common}

In this section, we collect together some results concerning
endo-classes of ps-characters and their relationship with common
approximations (see~\cite[\S8]{BKsemi}). Much of this is implicit
in~\cite{BKsemi} in the split case.

\subsection{} Let $(k,\b)$ be a simple pair and, for $i=1,2$, let $[\La_i,n_i,m_i,\h_i(\b)]$ be a realization in a simple central $\F$-algebra $\A_i$. According to~\cite[\S3.3]{VS1}, there is a canonical \emph{transfer} map
$$
\tau:\Cc(\La_1,m_1,\h_1(\b)) \to \Cc(\La_2,m_2,\h_2(\b)).
$$
Denote by $\boldsymbol{\EuScript{C}}_{(k,\b)}$ the set of pairs 
$([\La,n,m,\h(\b)],\t)$ made of a realization $[\La,n,m,\h(\b)]$ 
of $(k,\b)$ in a simple central $\F$-algebra and a simple character
$\t\in\Cc(\La,m,\h(\b))$. Then the transfer maps $\tau$ induce an equivalence relation on $\boldsymbol{\EuScript{C}}_{(k,\b)}$.

\begin{defi} 
A \emph{potential simple character} over $\F$ (or \emph{ps-character}
for short) is a triple $(\Theta,k,\b)$ made of a simple pair $(k,\b)$
over $\F$ and an equivalence class $\Theta$ in 
$\boldsymbol{\EuScript{C}}_{(k,\b)}$.  
\end{defi}

When the context is clear, we will often denote by $\Theta$ the 
ps-character $(\Theta,k,\b)$.
Given a realization $[\La,n,m,\h(\b)]$ of $(k,\b)$, we will denote 
by $\Theta(\La,m,\h)$ 
the simple character $\t$ such that the pair 
$([\La,n,m,\h(\b)],\t)$ belongs to $\Theta$.

\begin{defi} 
For $i=1,2$, let $(\Theta_{i},k_{i},\b_{i})$ be a ps-character over $\F$. We say that these ps-characters are \emph{endo-equivalent}, denoted:
$$
\Theta_1\thickapprox\Theta_2,
$$
if $k_1=k_2$ and 
$[\F(\b_1):\F]=[\F(\b_2):\F]$, and if there exist a 
simple central $\F$-algebra $\A$ and realizations
$[\La,n_i,m_i,\h_i(\b_i)]$ of $(k_i,\b_i)$ in $\A$, 
for $i=1,2$, such that the simple characters 
$\Theta_1(\La,m_1,\h_1)$ and $\Theta_2(\La,m_2,\h_2)$ intertwine in 
$\mult\A$.
\end{defi}

That this defines an equivalence relation on ps-characters follows from a major result in~\cite{BSS}:

\begin{prop}[{\cite[Theorem~1.11]{BSS}}]
\label{prop:IICpschars}
For $i=1,2$, let $(\Theta_{i},k_{i},\b_{i})$ be a ps-character over $\F$, 
and suppose that $\Theta_{1}\thickapprox\Theta_{2}$. 
Let $\A$ be a simple central $\F$-algebra and let 
$[\La,n_i,m_i,\h_i(\b_i)]$ be realizations of 
$(k_{i},\b_i)$ in $\A$, for $i=1,2$. 
Then $\t_1=\Theta_1(\La,m_1,\h_1)$ and 
$\t_2=\Theta_2(\La,m_2,\h_2)$ intertwine in $\A^{\times}$.
\end{prop}

In the situation of Proposition~\ref{prop:IICpschars}, if $(F[\h_i(\b_i)],\La)$ have the same embedding type then we can apply Proposition~\ref{prop:IICwithK} to conclude that the realizations $\t_1$, $\t_2$ are actually conjugate.

\medskip

We will use the common convention that, for each $t\ge 0$, there is the
\emph{trivial ps-character $\Theta^{(t)}_0$}, whose realization on any lattice
sequence $\La$ is the trivial character of the group $\U^{t+1}(\La)$;
then $\{\Theta^{(t)}_0\}$ forms a singleton equivalence class under
endo-equivalence.

\subsection{}
Let $[\La,n,0,\b]$ be a simple stratum in $\A=\End_\D(\V)$. 
Let $\t\in\Cc(\La,0,\b)$ be a simple stratum and denote by $(\T,0,\b)$
the ps-character that it determines -- that is,
$\t$ is a realization of $\T$.

For $t\ge 0$, let $[\La,n,t,\b^{(t)}]$ be a simple stratum equivalent
to the pure stratum $[\La,n,t,\b]$ and write
$\E^{(t)}=\F[\b^{(t)}]$. Then the restriction 
$\t|{\H^{t+1}(\b,\La)}$ is a simple character in
$\Cc(\La,t,\b^{(t)})$ and we denote by $(\T^{(t)},k^{(t)},\b^{(t)})$ the ps-character
determined by this restriction,  
with $k^{(t)}=\lfloor t/e(\La|\Oo_{\E^{(t)}})\rfloor$.

\medskip

{\bf Remark.} We allow the case $t\ge n$, in which case we interpret
the stratum $[\La,n,t,\b]$ to be equivalent to a null stratum
$[\La,t,t,0]$ and $\T^{(t)}$ is the trivial ps-character.

\begin{lemm}\label{lem:PSembed} 
With notation as above let $(\T_\g,k,\g)$ be a ps-character which is endo-equivalent to $\T^{(t)}$. Then
there is an embedding $\iota_\g:\F[\g]\hookrightarrow \A$ such that 
$$
\t|{\H^{t+1}(\b,\La)} = \T_\g(\La,t,\iota_\g).
$$
\end{lemm}

\begin{proof} Put $\E_\g=\F[\g]$. Since the ps-characters $\T_\g$,
$\T^{(t)}$ are endo-equivalent, the fields $\E_\g$, $\E^{(t)}$
have the same invariants by~\cite[Lemma~4.8]{BSS}: 
$$
e(\E_\g/\F)=e(\E^{(t)}/\F), \ \ 
f(\E_\g/\F)=f(\E^{(t)}/\F),
 \hbox{ \ and \ } 
k_\F(\g)=k_\F(\b^{(t)}).
$$
Then, by Lemma~\ref{lem:embedtype}, there is an embedding
$\iota_\g:\K_\g\hookrightarrow\A$ such that
$[\La,n,t,\iota_\g(\g)]$ is a pure stratum with 
the same embedding type as $[\La,n,t,\b^{(t)}]$, which is simple since
$k_\F(\g)=k_\F(\b^{(t)})$ and $[\La,n,t,\b^{(t)}]$ is simple. 

Finally, since the ps-characters are endo-equivalent,
by Propositions~\ref{prop:IICpschars} and~\ref{prop:IICwithK}, the realization
$\T_\g(\La,t,\iota_\g)$ is conjugate to
$\t|{\H^{t+1}(\b,\La)}$ by some $u\in\KK(\La)$. 
Conjugating our embedding $\iota_\g$ by $u$ gives
the desired embedding. 
\end{proof}

For $j=1,\ldots,r$, let $\A^j=\End_\D(\V^j)$, let $[\La^j,n_j,0,\b_j]$
be a simple stratum in $\A^j$, write $\E_j=\F[\b_j]$, and let 
$\t_j\in\Cc(\La^j,0,\b_j)$. We normalize so that the lattice sequences
$\La^j$ have the same $\Oo_\F$-period $e$.  

As above, for $t \ge 0 $, let $[\La^j,n_j,t,\b_j^{(t)}]$ be a simple
stratum equivalent to the pure stratum $[\La^j,n_j,t,\b_j]$, and write
$\E_j^{(t)}=\F[\b_j^{(t)}]$. The restriction
$\t_j|{\H^{t+1}(\b_j,\La^j)}$ is a simple character in
$\Cc(\La^j,t,\b_j^{(t)})$ and we denote by $(\T_j^{(t)},k_j^{(t)},\b_j^{(t)})$ the
ps-character determined by this restriction, 
with $k_j^{(t)}=\lfloor t/e(\La|\Oo_{\E_j^{(t)}})\rfloor$. 

\medskip

We put $\V=\V^1\oplus\cdots\oplus\V^r$, and set
$\La=\La^1\oplus\cdots\oplus\La^r$, a lattice sequence in $\V$ of
$\Oo_\F$-period $e$. Write $\A=\End_\D(\V)$ and denote by $\e^j$ the
idempotents in $\PP_0(\La)$ corresponding to the decomposition of
$\V$. We put $\b=\sum_{j=1}^r \e^j\b_j\e^j$. 
Then $[\La,n,0,\b]$ is a stratum in $\A$, with
$n=\max_j n_j$. We write $\LL$ for the stabilizer in
$\G=\Aut_\D(\V)$ of the decomposition $\V=\V^1\oplus\cdots\oplus\V^r$,
and $\A_\LL$ for its stabilizer in $\A$.

\begin{defi} A \emph{common approximation of $(\t_j)$ of
level $t$ on $\La$} is a pair $([\La,n,t,\g],\vartheta)$ consisting of: a simple stratum
$[\La,n,t,\g]$ with $\g\in\A_\LL$ and $0\le t\le n$, such that
$$
\H^{t+1}(\g,\La)\cap\LL  =\prod_{j=1}^r\H^{t+1}(\b_j,\La^j);
$$
and a simple character $\vartheta\in\Cc(\La,0,\g)$ such that
$$
\vartheta|\H^{t+1}(\g,\La)\cap\LL  = \t_1\otimes\cdots\otimes\t_r.
$$
\end{defi}

When we have such a common approximation, we will identify $\g$ with
its images $\e^j\g\e^j$ in $\A^j$.

\begin{lemm}\label{lem:commonPS} 
Let $0\le t\le n$. Then the following are equivalent:
\begin{itemize}
\item[(i)] There is a common approximation of $(\t_j)$ of
level $t$ on $\La$.
\item[(ii)] The ps-characters $\T_j^{(t)}$ are endo-equivalent.
\end{itemize} 
\end{lemm}

\begin{proof} (ii)$\Rightarrow$(i) Let $(\T_\g,k,\g)$ be a
ps-character which is
endo-equivalent to all $\T_j^{(t)}$. Then, by
Lemma~\ref{lem:PSembed}, for each $j$ there is an embedding 
$\ii_j:\F[\g]\hookrightarrow \A^j$ such that 
$\T_\g(\La^j,t,\ii_j)=\t_j|\H^{t+1}(\b_j,\La^j)$. Denote by $\ii$ the 
diagonal embedding $\ii:\F[\g]\hookrightarrow \bigoplus_{j=1}^r \A^j
\subseteq \A$ and let $\vartheta$ be any simple character in
$\Cc(\La,0,\ii(\g))$ which restricts to $\T_\g(\La,t,\ii)$ on
$\H^{t+1}(\g,\La)$. Then $([\La,n,t,\ii(\g)],\vartheta)$ is a common
approximation as required.

(i)$\Rightarrow$(ii) Suppose $([\La,n,t,\g],\vartheta)$ is a common
approximation of $(\t_j)$ of level $t$ on $\La$.
Then the characters $\t_j|\H^{t+1}(\b_j,\La^j)$ are
simple characters in $\Cc(\La^j,t,\g)$ and,
by~\cite[Th\'eor\`eme~2.17]{SS4}, these characters are all
transfers of each other relative to $\g$; hence the corresponding
ps-characters (which are 
supported by the simple pair $(k,\g)$, with $k=\lfloor
t/e(\La|\Oo_{\F[\g]})\rfloor$) are endo-equivalent.
\end{proof}

We suppose now that there is a common approximation
$([\La,n,t,\g],\vartheta)$ of $(\t_j)$. Denote by $\B_\g$ the
$\A$-centralizer of $\E_\g=\F[\g]$, by $\V_\g$ a simple left $\B_\g$-module,
by $\D_\g$ the opposite algebra to $\End_{\B_\g}(\V_\g)$, and by
$s_\g$ a tame corestriction on $\A$. Note that, since $\g\in\A_\LL$,
the restriction of $s_\g$ to $\A^j$ is also a tame corestriction
(see~\cite[Proposition~2.26]{SS4}). Also, the idempotents $\e^j$
lie in $\B_\g$ so correspond to a decomposition
$\V_\g=\V^1_\g\oplus\cdots\oplus\V^r_\g$.  
Let $\Ga_\g$ be an $\Oo_{\D_\g}$-lattice sequence in $\V_\g$ such that
$\PP_n(\La)\cap\B_\g=\PB_n(\Ga_\g)$, for all $n\in\Z$, and put
$\Ga_\g^j=\Ga_\g\cap\V_\g^j$, for $1\le j\le r$.

Since $\t_j$ and $\vartheta$ coincide on $\H^{t+1}(\g,\La)$,
Corollary~\ref{cor:derived} says that there is 
$c_j\in\PP_{-t}(\La^j)$ 
such that $\t_j|\H^{t}(\beta_j,\La^j) = \vartheta\psi_{c_j}$, and that
the derived stratum $[\Ga^j_\g,t,t-1,s_\g(c_j)]$ is equivalent to a
simple (or null) stratum. The following result is a generalization of
Corollary~\ref{cor:minpolpair}.

\begin{coro}\label{cor:endocommon}
In the situation above, the derived strata
$[\Ga^j_\g,t,t-1,s_\g(c_j)]$ have the same minimum polynomial
if and only if the ps-characters $\T_j^{(t-1)}$ are endo-equivalent.
\end{coro}

\begin{proof} Suppose first that the minimum polynomials of the derived strata
$[\Ga^j_\g,t,t-1,s_\g(c_j)]$ are all the same. Note that they are
irreducible since these strata are equivalent to simple strata. 
Then the derived stratum $[\Ga_\g,t,t-1,s_\g(c)]$ is equivalent to
a simple stratum by Corollary~\ref{cor:minpol} and,
by Proposition~\ref{prop:refinement}, there is a simple stratum
$[\La,n,t-1,\g']$ equivalent to $[\La,n,t-1,\g+c]$, 
so that $\vartheta\psi_c\in\Cc(\La,t-1,\g')$. Then, for any
$\vartheta'\in\Cc(\La,0,\g')$ extending $\vartheta\psi_c$, the pair
$([\La,n,t-1,\g'],\vartheta')$ is a common
approximation of level $t-1$. Hence, by Lemma~\ref{lem:commonPS}, the
ps-characters $\T_j^{(t-1)}$ are endo-equivalent.

Conversely, suppose the ps-characters $\T_j^{(t-1)}$ are
endo-equivalent so, by Lemma~\ref{lem:commonPS}, there is a common
approximation $([\La,n,t-1,\g'],\vartheta')$ of level $t-1$. We then
have $\vartheta'=\vartheta\psi_c$ and, by Corollary~\ref{cor:derived},
the derived stratum
$[\Ga_\g,t,t-1,s_\g(c)]$ is equivalent to a simple (or null)
stratum. Hence, by Corollary~\ref{cor:minpol}, the derived strata
$[\Ga^j_\g,t,t-1,s_\g(c_j)]$ have the same minimum polynomial. 
\end{proof}


\section{Simple types}\label{S.simple}

In this section we recall some results from ~\cite{VS3} concerning
simple types. In later sections we will need these in slightly more generality than in \emph{op.\ cit.} -- in
particular, in the case where we have a non-strict lattice
sequence. Already in the case of $\GL(n,\F)$, simple types on non-strict lattice sequence are required in \cite{BKsemi}, although this is not immediately apparent. The proofs are mostly identical to those in~\cite{VS2,VS3}.

\subsection{}

Let $[\La,n,0,\b]$ be a simple stratum in $\A=\End_\D(\V)$, and use all
the usual notation from the previous sections. Since $\b$ is fixed, we
will omit it from the notations; when $\La$ is fixed, we will omit that
also. 

\begin{lemm}\label{eta}
 Let $\t\in\Cc(\La,0,\b)$ be a simple character. Then
there is a unique irreducible representation $\eta$ of $\J^1$
which contains $\t$; moreover, $\eta|\H^1$ is a multiple of
$\t$, the dimension $\dim(\eta)=(\J^1:\H^1)^{1/2}$
and 
$$
\dim \I_g(\eta)=\begin{cases} 1 
&\hbox{ if }g\in\J^1\B^\times\J^1, \\
0 &\hbox{ otherwise.} \end{cases}
$$
\end{lemm}

\begin{proof} The proof of all but the final assertion is identical to
  that of~\cite[Proposition~5.1.1)]{BK},
  replacing~\cite[Theorem~3.4.1)]{BK}
  by~\cite[Proposition~2.31]{SS4}. The proof of the final assertion is
  identical to that of~\cite[Proposition~5.1.8)]{BK}, replacing the
  exact sequences there by those of~\cite[Proposition~2.27]{SS4}.
\end{proof}

\begin{lemm} 
For $i=1,2$, let $[\La^i,n_i,0,\b]$ be a simple stratum in $\A$, let
$\t_i\in\Cc(\La^i,0,\b)$, and let $\eta_i$ be the unique
irreducible representation of $\J^1(\b,\La^i)$ which contains
$\t_i$. Then
$$
\frac{\dim(\eta_1)}{\dim(\eta_2)} =
\frac{(\J^1(\b,\La^1):\J^1(\b,\La^2))}{(\U^1(\La^1)\cap\B:\U^1(\La^2)\cap\B)}.
$$
\end{lemm}

\begin{proof} Again, the proof is identical to that
  of~\cite[Proposition~5.1.2]{BK}, replacing the exact sequence there
  with~\cite[Proposition~2.27]{SS4}. 
\end{proof}

\subsection{}
Recall that a \emph{$\b$-extension of $\t$} is a representation
$\k$ of $\J$ which extends the representation $\eta$ given by
Lemma~\ref{eta} and such that $I_G(\k)\supset \B^\times$. In the case
that $\La$ is strict, the existence of $\b$-extensions is given
by~\cite[Th\'eor\`eme~2.28]{VS2}. Using this, we proceed here via a
simplified version of the compatibility argument used there.

\begin{defi}\label{def:compat}
Let $[\La,n,0,\b]$ be a simple stratum in $\A$ and let
$\La'$ be an $E$-pure lattice sequence in $\V$ such that
$\PP_0(\La)=\PP_0(\La')$. Let $\t,\t'$ be simple characters
which are realizations of the same ps-character on $\La,\La'$
respectively, and let $\k,\k'$ be extensions of the representations
$\eta,\eta'$ given by Lemma~\ref{eta} respectively. We say that
$\k,\k'$ are \emph{compatible} (or \emph{mutually coherent}) if
$$
\Ind_{\J}^{(\U(\La)\cap\B)\U^1(\La)} \k \simeq 
\Ind_{\J(\b,\La')}^{(\U(\La)\cap\B)\U^1(\La)} \k'.
$$
\end{defi}

\begin{prop} With the notations of Definition~\ref{def:compat}, the
  notion of compatibility induces a bijection
$$
\left\{\hbox{$\b$-extensions of $\t$}\right\} \longleftrightarrow
\left\{\hbox{$\b$-extensions of $\t'$}\right\}.
$$
In particular, there is a $\b$-extension $\k$ of $\t$, and then
the set of $\b$-extension of $\t$ is given by
$$
\left\{\k\otimes\left(\chi\circ\N_{\B/\E}\right) : \chi\in
\widehat{\U_\E/\U^1_\E}\right\}.
$$
\end{prop}

\begin{proof}
The first assertion follows as in~\cite[Lemmes~2.23,2.24]{VS2}
(cf. also~\cite[Proposition~5.25]{BK}). Now taking $\La'$ to be the
strict lattice sequence in $\V$ with the same image as $\La$, the final
assertion follows from~\cite[Th\'eor\`eme~2.28]{VS2}.
\end{proof}

\subsection{}\label{SS.etaP}
We continue with a simple stratum $[\La,n,0,\b]$ and a simple character
$\t\in\Cc(\La,0,\b)$, together with the unique irreducible
representation $\eta$ of $\J^1$ containing $\t$. Let $\V_\E$  be a
simple left $\B$-module, let $\D_\E$ be the opposite algebra to
$\End_\B(\V_\E)$, and set $m_\E=\dim_{\D_\E}\V_\E$. We write $\Ga$ for the unique (up to translation) $\Oo_{\D_\E}$-lattice sequence on $\V_\E$ such that $\PP_k(\La)\cap\B=\PP_k(\Ga)$, for all $k\in\ZZ$.
 
We suppose given a decomposition $\V=\V^1\oplus\cdots\oplus\V^r$
which is \emph{subordinate to $\PP_0(\Ga)$} in the sense
of~\cite[D\'efinition~5.1]{SS4}: that is, it is a decomposition of
$\E\otimes\D$-bimodules and, writing $\e^j$ for the
idempotents of $\PP_0(\Ga)$ defined by the decomposition
and $m_j=\dim_{\D_\E}\e^j\V_\E$, there is an isomorphism of
$E$-algebras $\Psi:\B\to\M_{m_\E}(\D_\E)$ such that:
\begin{enumerate}
\item[(i)] for $1\le j\le r$, the idempotent $\Psi(\e^j)$ is
  $\I^j=\diag(0,\ldots,\Id_{m_j},\ldots,0)$:
\item[(ii)] The hereditary order $\Psi(\PB_0(\Ga))$ is the
  $\Oo_\E$-subalgebra of $\M_{m_\E}(\Oo_{\D_\E})$ consisting of
  matrices whose reduction modulo $\p_{\D_\E}$ is upper
  triangular by blocks of size $(m_1,\ldots,m_r)$.
\end{enumerate} 
Note then that $\La^j=\La\cap\V^j$ is in the affine class of a strict lattice
sequence of $\Oo_{\D_\E}$-period $1$ in $\V^j$. 

Let $\P$ be the parabolic subgroup of $\G$ stabilizing the flag
$$
\{0\}\subset \V^1\subset \V^1\oplus \V^2 \subset \cdots \subset \V,
$$
and write $\P=\LL\N$, where $\LL$ is the stabilizer of the decomposition
$\V=\bigoplus_{j=1}^r \V^j$ and $\N$ is the unipotent radical. Write
$\P_-=\LL\N_-$ for the opposite parabolic relative to $\LL$.

We define the groups
$$
\J_\P=\H^1\left(\J\cap\P\right),\quad
\J^1_\P=\H^1\left(\J^1\cap\P\right),\quad
\H^1_\P=\H^1\left(\J^1\cap\N\right),
$$
and define the character $\t_\P$ of $\H^1_\P$ by
$\t_\P(hu)=\t(h)$, for $h\in\H^1$ and
$u\in\J^1\cap\N$. We also put $\J_\LL=\J\cap\LL$ and
$\J^1_\LL=\J^1\cap\LL$ and notice that, since the decomposition is
subordinate to $\PP_0(\La)\cap\B$, we have
$$
\J_\P/\J^1_\P \simeq \J_\LL/\J^1_\LL \simeq
\U(\Ga)/\U^1(\Ga) \simeq \J/\J^1.
$$
In particular, given a representation of $\U(\Ga)$ trivial on
$\U^1(\Ga)$, we can also regard it as a
representation of $\J_\P$ (respectively $\J_\LL$, $\J$) trivial on $\J^1_\P$
respectively $\J^1_\LL$, $\J^1$).

The following Proposition summarizes the results
of~\cite[Propositions~5.3--5]{SS4} (see also \emph{op. cit.} \S5.8);
the results there are in the case 
that $\La$ is strict but, given our preliminary results above, identical
proofs apply in the general case. 

\begin{prop}
Let $\eta_\P$ denote the natural representation of $\J^1_\P$ on on the
$\J\cap\N$-invariants of $\eta$. Then $\eta_\P$ is the unique
irreducible representation of $\J^1_\P$ which contains
$\t_\P$. Moreover, $\Ind_{\J_\P}^{\J^1}\eta_\P$ is isomorphic
to $\eta$ and 
$$
\dim \I_g(\eta_\P)=\begin{cases} 1 
&\hbox{ if }g\in\J_\P^1\B^\times\J_\P^1, \\
0 &\hbox{ otherwise.} \end{cases}
$$
\end{prop}

\subsection{}
Now let $\k$ be a $\b$-extension of $\t$ and let $\k_\P$ denote
the natural representation of $\J_\P$ on on the $\J\cap\N$-invariants
of $\k$. The proof of~\cite[Proposition~5.8]{SS4} (see also
\emph{op. cit.} \S5.8) again generalizes to
the non-strict case and gives:

\begin{prop}\label{prop:kP}
$\k_\P$ is an irreducible representation of $\J_\P$ with
the following properties:
\begin{enumerate}
\item[(i)] $\k_\P|\H^1(\b,\La)$ is a multiple of $\t$;
\item[(ii)] $\k_\P$ is trivial on $\J_\P\cap\N$ and $\J_\P\cap\N_-$;
\item[(iii)] $\k_\P|\J_\LL \simeq \k_1\otimes\cdots\otimes\k_r$, for
some $\b$-extensions $\k_j$ containing $\t_j$;
\item[(iv)] $\I_\G(\k_\P)=\J_\P\B^\times\J_\P$;
\item[(v)] if $m_j=m_k$ then $\k_j\simeq\k_k$. 
\item[(vi)] if $\xi$ is an irreducible representation of
$\U(\Ga)$ trivial on $\U^1(\Ga)$, then 
$$
\Ind_{\J_\P}^\J (\k_\P\otimes\xi) \simeq \k\otimes\xi,
\qquad\hbox{and}\qquad
\I_\G(\k_\P\otimes\xi)=\J_\P\I_{\B^\times}(\xi)\J_\P.
$$
\end{enumerate}
\end{prop}

\begin{proof} The only property not given
  by~\cite[Proposition~5.8~or~\S5.8]{SS4} is (v), where
imitate the proof of~\cite[Corollary~7.2.6]{BK}. There is a
permutation matrix in $w\in\M_{m_\E}(\D_\E)$ (which we have identified with
$\B$ via $\Psi$ above) which swaps $\V^j$ with $\V^k$, and leaves all
other $\V^i$ fixed. In particular, it lies in $\B$ 
and normalizes $\J_\P\cap\LL$. 
Then $w$ intertwines $\k_\P$ so normalizes $\k_\P|\J_\LL$ and
hence induces an isomorphism between $\k_j$ and $\k_k$. 
\end{proof}

\subsection{}
Finally, we consider the case where all $m_j$ are equal to some
integer $s$, so that $\U(\Ga)/\U^1(\Ga)\simeq
\GL_s(k_{\D_\E})^r$ and let $\xi$ be the inflation to $\U(\Ga)$
of the representation $\sigma^{\otimes r}$, for $\sigma$ an
irreducible cuspidal representation of $\GL_s(k_{\D_\E})$. 

We put $\l=\k\otimes\xi$, $\l_\P=\k_\P\otimes\xi$ and
$\l_\LL=\l_P|\J_\LL$. 

\begin{prop}\label{simplecase}
The pair $(\J_\P,\l_\P)$ is a cover of
$(\J_\LL,\l_\LL)$
and
$$
\Hh(\G,\l_\P)\cong \Hh(r,q_{\D_\E}^{s}).
$$
\end{prop}

We remark that the parameter $q_{\D_\E}^{s}$ here is the same as that
in Theorem~\ref{thm:simplecover} of the introduction, by~\cite[Theorem~4.6]{VSU0}.

\begin{proof}
In the case that $\La$ is strict, this is given
by~\cite[Proposition~5.5, Th\'eor\`eme~4.6]{VS3}. The idea of the
proof is to reduce to this case, as in
the proof of~\cite[Th\'eor\`eme~5.6]{VS3}. Let $\Ll$
be the strict lattice sequence with the same image as $\La$ and we make the same
constructions for $\Ll$, which we denote with the superscript
$^\Ll$. In particular, we choose a $\b$-extension $\k^\Ll$ compatible
with $\k$. Hence (as in~\cite[Proposition~4.5]{VS3}) we have a
support-preserving isomorphism
$$
\Hh(\G,\l) \cong \Hh(\G,\l^\Ll).
$$
Moreover, we have $\l=\Ind_{\J_\P}^\J \l_\P$ and 
$\l^\Ll=\Ind_{\J^\Ll_\P}^{\J^\Ll} \l^\Ll_\P$, so we get a
support-preserving isomorphism
$$
\Hh(\G,\l_\P) \cong \Hh(\G,\l^\Ll_\P).
$$
Then the assertions follow from (the proof of)~\cite[Proposition~5.5]{VS3}.
\end{proof}

\begin{defi} A pair $(\J,\l)$ as in this paragraph is called a \emph{simple type}; if $r=1$ then it is called a \emph{maximal simple type}.
\end{defi}

Recall that the main result of~\cite{SS4} (Th\'eor\`emes~5.21,~5.23) is
that every irreducible cuspidal representation of $\G$ contains a
maximal simple type.


\section{Intertwining and conjugacy}\label{S.IIC}

In this section we consider the unicity of the simple type contained in an irreducible representation $\pi$ of $\G=\GL_m(\D)$. That is, we suppose the inertial class $\ss(\pi)$ of $\pi$ is homogeneous: there are a positive integer $r$ dividing $m$, an irreducible cuspidal representation $\rho$ of the group $\G_0=\GL_{m/r}(\D)$ and unramified characters $\chi_i$ of $\G_0$, with $i\in\{1,\dots,r\}$, such that $\pi$ is isomorphic to a quotient of the normalized parabolically induced  representation $\rho\chi_1\times\dots\times\rho\chi_r$.
Unlike the situation for $\D=\F$, the simple type is \emph{not} uniquely determined up to conjugacy in general, as there is a galois action we must take into account.

We consider the set $\SS$ of \emph{sound simple types}
$$
\SS=\left\{\,\hbox{\begin{minipage}{0.6\textwidth}simple types $(\J(\b,\La),\l)$ such that $\La$ is strict, $\PP_0(\La)$ is principal and $(\F[\b],\La)$ is soundly embedded\end{minipage}}\,\right\}.
$$
For $(\J,\l)\in\SS$, we use all the associated notation of~\S\ref{S.simple}: that is, there are a (sound) simple stratum $[\La,n,0,\b]$, a simple character $\t\in\Cc(\La,0,\b)$, a $\b$-extension $\k$ and a representation $\xi$ which is the inflation to $\U(\Ga)$ of the representation $\overline\xi=\s^{\otimes r}$ of
$$
\U(\Ga)/\U^1(\Ga) \simeq \GL_s(k_{\D_\E})^r,
$$
for $\s$ an irreducible cuspidal representation of $\GL_s(k_{\D_\E})$.
Note that, implicit in the isomorphism above are the choice of a decomposition $\V=\V^1\oplus\cdots\oplus\V^r$ subordinate to $\PP_0(\Ga)$ and the choice of an $\E$-algebra isomorphism $\Psi:\B\to\M_{m_\E}(\D_\E)$ as in paragraph~\ref{SS.etaP}.

We fix a  
uniformizer $\w$ of $\D_\E$. The Galois group $\Gg=\Gal(k_{\D_\E}/k_\E)$ identifies, via reduction, with the group generated by ${\rm Ad}(\w)$, the inner automorphism given by conjugation by $\w$.
The Galois group $\Gg$ acts on the representations of $\GL_s(k_{\D_\E})$. Moreover, a different choice of $\E$-algebra isomorphism $\Psi$ could result in a change in the identification $\U(\Ga)/\U^1(\Ga) \simeq \GL_s(k_{\D_\E})^r$ by conjugating each factor by an element of $\Gg$, rather than just by an inner automorphism. Thus we define $[\s]$ to be the orbit of $\s$ under the action of $\Gg$ and set
$$
[\l] = \left\{\,\hbox{\begin{minipage}{0.58\textwidth}equivalence classes of representations $\k\otimes\xi'$, for $\xi'$ the inflation to $\U(\Ga)$ of $\s_1\otimes\cdots\otimes\s_r$, with $\s_i\in[\s]$\end{minipage}}\,\right\}.
$$

We also define an equivalence relation on $\SS$ by: $(\J,\l) \sim (\J,\l')$ if and only if there is an irreducible representation $\pi$ of $\G$ such that $\pi$ contains both $\l$ and $\l'$.

\begin{theo}[{cf.~\cite[Theorem~5.7.1]{BK}}]\label{thm:simpleconj} 
Let $(\J,\l)$ and $(\J',\l')$ be sound simple types. Then $(\J,\l) \sim (\J',\l')$ if and only if there exists $g\in\G$ such that $^g\J'=\J$ and $\left[{}^g\l'\right]=[\l]$.
\end{theo}

\begin{proof} The proof follows that of~\cite[Theorem~5.7.1]{BK}. Suppose first that $(\J,\l) \sim (\J',\l')$ and use all notation as above, with a prime $'$ to indicate the corresponding objects for $\J',\l')$, in particular writing $\E'=\F[\b']$. We also write $\Theta$ for the ps-character defined by $\t$. Then~\cite[Theorems~9.2,~9.3]{BSS} imply that:
\begin{itemize}
\item[$\bullet$] $\Theta'$ is endo-equivalent to $\Theta$;
\item[$\bullet$] $(\E,\La)$ and $(\E',\La')$ have the same embedding type.
\end{itemize}
In particular, by Propositions~\ref{prop:IICpschars} and~\ref{prop:IICwithK}, and the definition of embedding type, there is $g\in\G$ such that $g\La'$ is in the translation class of $\La$, $^g\H^1(\b',\La')=\H^1(\b,\La)$ and $^g\t'=\t$. Replacing $\l'$ by $^g\l'$ we may assume that $g=1$, so that $\t'=\t$; moreover, since changing $\La$ in its translation class affects nothing, we may assume $\La'=\La$.

Now the $\U(\La)$ intertwining of $\t$ is $\J(\b,\La)$ so we get $\J'=\J$ and $\J^1(\b',\La')=\J^1(\b,\La)$. By unicity in Lemma~\ref{eta}, we get $\eta'=\eta$. Moreover, since the intertwining of $\t$ is $\J\B^\times\J=\J(\B')^\times\J$, the $\b$-extension $\k$ is also a $\b'$-extension and we may assume $\k'=\k$.

As in the proof of~\cite[Theorem~5.7.1]{BK}, the cuspidality of $\overline\xi$ can be interpreted in purely group-theoretic terms. In particular, if we identify $\J/\J^1$ with $\GL_s(k_{\D_\E})^r$, then $\overline{\xi'}$ decomposes as $\s'_1\otimes\cdots\otimes\s'_r$ with all $\s'_i$ cuspidal.

Now $\l$, $\l'$ are contained in some irreducible representation $\pi$ of $\G$ and are therefore intertwined by some $x\in\G$. Since $\l$, $\l'$ both restrict to a multiple of $\t$ on $\H^1$, we have also that $x$ intertwines $\t$ and thus $x\in\J^1\B^\times\J^1$. In particular, we may assume $x\in\B^\times$. Then, arguing as in~\cite[Proposition~5.3.2]{BK}, we see that $x$ intertwines $\xi$ with $\xi'$, when we interpret them as representations of $\U(\Ga)$.

To finish, we argue again as in the proof of~\cite[Theorem~5.7.1]{BK}, using results from~\cite{GSZ}. In particular, we will use some notation from~\cite[\S0.8]{GSZ}, writing $\tilde\W_\B$ for the generalized affine Weyl group in $\B^\times$, which we have identified with $\GL_{m_\E}(\D_\E)$ via $\Psi$. By the affine Bruhat decomposition, we may assume $x\in\tilde\W_\B$.

Since $\overline\xi$ and $\overline{\xi'}$ are cuspidal, the same proof as that of~\cite[Proposition~1.2]{GSZ} shows that $x$ normalizes $\LL\cap\U(\Ga)$, where $\LL$ is the Levi subgroup of $\G$ which is the stabilizer of the decomposition $\V=\V^1\oplus\cdots\oplus\V^r$. Likewise,~\cite[Lemma~1.5]{GSZ} implies that 
$$
\Hom_{\U(\Ga)\cap x^{-1}\U(\Ga)x}(\xi',\xi^x) = \Hom_{\U(\Ga)/\U^1(\Ga)}(\overline{\xi'},\overline\xi^x).
$$
Now $\overline\xi^x=\s_1\otimes\cdots\otimes\s_r$, with each $\s_i\in[\s]$ and, since $\overline{\xi'}$, $\overline\xi^x$ are irreducible, we deduce $\s'_i\simeq\s_i\in[\s]$. Hence the equivalence class of $\l'$ is in $[\l]$, so $[\l']=[\l]$, as required.

The converse is given by~\cite[Proposition~5.19]{SS4}. 
\end{proof}

Although, in general, equivalent sound simple types are not conjugate, they are in the special case of cuspidal representations:

\begin{coro}\label{cor:maxsimpleconj}
Suppose $(\J,\l)$, $(\J',\l')$ are maximal simple types. Then $(\J,\l) \sim (\J',\l')$ if and only if there exists $g\in\G$ such that $^g\J'=\J$ and $^g\l'\simeq\l$.
\end{coro}

\begin{proof} Suppose $(\J,\l)$ and $(\J',\l')$ are equivalent maximal simple types. By Theorem~\ref{thm:simpleconj}, there there exists $g\in\G$ such that $^g\J'=\J$ and $\left[{}^g\l'\right]=[\l]$. That is, as in the proof of Theorem~\ref{thm:simpleconj}, we can write $^g\l'\simeq \k\otimes\xi'$, with $\overline{\xi'}\simeq\overline\xi^\g$, for some $\g\in\Gg$. But the action of $\g$ can be realized as conjugation by a power of $\w$, which normalizes $\Ga$, so there is $y\in\KK(\Ga)$ such that $^y\xi'\simeq\xi$. Since $y$ also normalizes $\k$, we deduce that $^{yg}\l'\simeq\l$.
\end{proof}

This also completes the proof of Theorem~\ref{thm:cuspidal} of the introduction.


\section{Semisimple types}\label{S.Semisimple}

Suppose we have cuspidal representations $\pi_j$ of
$\G^j=\Aut_\D(\V^j)$, for $1\le j\le r$, and we think of
$\LL=\prod_{j=1}^r\G_j$ as a Levi subgroup in $\G=\Aut_\D(\V)$, where 
$\V=\bigoplus_{j=1}^r \V^j$. The aim of this section is to prove
Theorem~\ref{thm:coverhecke}: a maximal simple type for
$(\LL,\pi_1\otimes\cdots\otimes\pi_r)$ admits a cover, with an
explicitly computable Hecke algebra.

For $j=1,\ldots,r$, the cuspidal representation $\pi_j$ contains a
(maximal) simple type $(\J(\b_j,\La^j),\l_j)$, where  
$[\La^j,n_j,0,\b_j]$ is a simple stratum in $\A^j=\End_\D(\V^j)$, 
$\t_j$ is a simple character of $\H^1(\b_j,\La^j)$, $\k_j$ is a
$\b_j$-extension containing $\t_j$, $\sigma_j$ is a cuspidal
representation of $\J(\b_j,\La^j)/\J^1(\b_j,\La^j)$ and
$\l_j=\k_j\otimes\sigma_j$. We write $(\T_j,0,\b_j)$ for the ps-character
defined by $\t_j$. Then our type in $\LL$ is $(\J_\LL,\l_\LL)$,
given by 
$$
\J_\LL = \prod_{j=1}^r \J(\b_j,\La^j),\quad
\l_\LL = \l_1\otimes\cdots\otimes\l_r.
$$

For $j=1,\ldots,r$, we write $\B^j$ for the $\A_j$-centralizer of
$\b_j$. Then $\B_j$ has the form $\End_{\D_{\E_j}}(\W^j)$, for some
right $\D_{\E_j}$-vector space $\W^j$. 
We write $\Ga^j$ for the unique strict $\Oo_{\D_{\E_j}}$-lattice
sequence in $\W^j$ such that $\PB_0(\Ga^j)=\PP_0(\La^j)\cap\B^j$. Since
$\pi_j$ is cuspidal, $\Ga^j$ is a sequence of $\Oo_{\D_{\E_j}}$-period $1$ and
then the normalizer in $\G^j\cap\B^j$ of $\Ga^j$ is just
$\KK(\La^j)\cap\B^j$; moreover, $\La^j$ is a strict
lattice sequence and $\PP_0(\La^j)$ is a
principal order in $\A^j$.


\subsection{The homogeneous case}

We suppose first that the $(\T_j,0,\b_j)$ are all endo-equivalent to
some fixed ps-character $(\T,0,\b)$. 

By Lemma~\ref{lem:PSembed}, for each $j$ there is a realization
$\T(\La^j,0,\ii_j)$ equal to $\t_j$. Hence we may (and do) assume that
all $\t_j$ are defined relative to the same simple pair $(0,\b)$
and are realizations of the same ps-character $\T$.

\medskip

We have $\V=\bigoplus_{j=1}^r \V^j$ and put $\W=\bigoplus_{j=1}^r
\W^j$, so that $\B=\End_{\D_\E}(\W)$. Write $\e^j$ for the idempotent
in $\B$ with image $\W^j$ and kernel $\bigoplus_{i\ne j}\W^i$. As an
element of $\A$ it has image $\V^j$ and kernel $\bigoplus_{i\ne
  j}\V^i$. 

Let $\Ga$ be an $\Oo_{\D_{\E}}$-lattice sequence in $\W$ such that:
\begin{enumerate}
\item[(i)] $\Ga\cap\W^j$ is in the affine class of $\Ga^j$, and 
\item[(ii)] $\Ga$ is subordinate to the decomposition
  $\W=\bigoplus_{j=1}^r \W^j$. 
\end{enumerate}
Note that condition (ii) is generically satisfied: that is, amongst
all lattice sequences satisfying (i), those also satisfying (ii) are
dense (in the building of $\B^\times$). A particular example of such a
sequence is given by
$$
\Ga(k)\ =\ \bigoplus_{j=1}^{r} 
\Ga^j\left(\left\lfloor\frac{k+j}{r}\right\rfloor\right),\qquad k\in\ZZ,
$$
which is strict of $\Oo_{\D_{\E}}$-period $r$ but not principal in
general. 

We fix $\Ga$ satisfying (i),(ii) and let $\La$ be the corresponding
$\Oo_\D$-lattice sequence in $\V$ given
by~\cite[Th\'eor\`eme~1.7]{SS4}. Then 
$\e^j\in\PB_0(\Ga)\subseteq\PP_0(\La)$ so $\La$ is decomposed by (indeed
subordinate to) the decomposition
$\V=\bigoplus_{j=1}^r \V^j$ and the lattice sequence $r\mapsto
\La(r)\cap \V^j$ is in the affine class of $\La^j$. In fact,
replacing $\La^j$ by this sequence changes nothing in the construction
of the type $(\J(\b_j,\La^j),\l_j)$ so we may (and do) assume this
done, in which case $\La=\bigoplus_{j=1}^r\La^j$.

Let $\P$ be the parabolic subgroup of $\G$ stabilizing the flag
$$
\{0\}\subset \V^1\subset \V^1\oplus \V^2 \subset \cdots \subset \V,
$$
and write $\P=\LL\N$, where $\LL$ is the stabilizer of the decomposition
$\V=\bigoplus_{j=1}^r \V^j$ and $\N$ is the unipotent radical. Write
$\P_-=\LL\N_-$ for the opposite parabolic relative to $\LL$.

Now $[\La,n,0,\b]$ is a simple stratum in $\A=\End_\D(\V)$, for a
suitable integer $n$, and we have 
$\H^1(\b,\La)\cap\LL\cong \prod_{j=1}^r \H^1(\b,\La^j)$, with similar
decompositions for $\J^1$ and $\J$. Moreover,
$\t=\T(V,0,\La)$ is a simple character such that
$$
\t|\H^1(\b,\La)\cap\LL = \t_1\otimes\cdots\otimes\t_r.
$$
As in~\S\ref{S.simple}, we define the groups
$$
\J_\P=\H^1(\b,\La)\left(\J(\b,\La)\cap\P\right),\qquad
\J^1_\P=\H^1(\b,\La)\left(\J^1(\b,\La)\cap\P\right),
$$
noting that $\J_\LL=\J_\P\cap\LL$.

Let $\k_\P$ be the irreducible representation of $\J_\P$ given by
Proposition~\ref{prop:kP}, so that $\k_\P|\J_\LL \simeq
\bigotimes_{j=1}^r \k'_j$, for some $\b$-extensions $\k'_j$ containing
$\t_j$. Then we can choose the decompositions
$\l_j=\k_j\otimes\sigma_j$ of the maximal simple types above so that
$\k_j=\k'_j$, which we assume done.

We define an equivalence relation $\sim$ on $\{1,\ldots,r\}$
by 
$$
j\sim k\ \iff\ \sigma_j\simeq\sigma_k^\gamma,\hbox{ for some
}\gamma\in\Gal(k_{\D_\e}/k_\E),
$$
and denote by $I_1,\ldots,I_l$ the equivalence classes. Put $r_i=\#
I_i$ and define $s_i$ by
$\J(\b,\La^j)/\J^1(\b,\La^j)\simeq\GL_{s_i}(k_{\D_\E})$, for any $j\in
I_i$. Note also that, by conjugating the types $\l_j$ by a suitable
element of $\D_\E^\times$, we may (and do)
assume that $\sigma_j\simeq\sigma_k$ whenever $j\sim k$.
Put $\Y^i=\bigoplus_{j\in I_i}\V^j$ and denote by $\M$ the Levi subgroup
which stabilizes the decomposition $\V=\bigoplus_{i=1}^l \Y^i$. Note
that this is the Levi subgroup $\M$ of the introduction.

\medskip

Now $\J_\P/\J^1_\P\cong\prod_{j=1}^r \J(\b,\La^j)/\J^1(\b,\La^j)$ so we
can define a representation $\sigma$ of $\J_\P$ inflated from
$\bigotimes_{j=1}^r \sigma_j$. Then we put $\l_\P=\k_\P\otimes\sigma$.
We put $\J_\M=\J_\P\cap \M$ and $\l_\M=\l_\P|\J_\M$. 

\begin{prop}\label{prop:hom}
The pair $(\J_\P,\l_\P)$ is a cover of
$(\J_\M,\l_\M)$, which is a cover of $(\J_\LL,\l_\LL)$. Moreover, we
have a support-preserving Hecke algebra isomorphism 
$$
\Hh(\M,\l_\M)\cong \Hh(\G,\l_\P)
$$
and
$$
\Hh(\M,\l_\M)\cong\bigotimes_{i=1}^l \Hh(r_i,q_{\D_\E}^{s_i}).
$$
\end{prop}

\begin{proof} The proof of the first assertion is identical to that
  of~\cite[Proposition~5.17]{SS4}; 
indeed the proof there shows that $\I_\G(\l_\P)\subseteq
\J_\P(\B^\times\cap\M)\J_\P$ and then the first Hecke algebra isomorphism
also follows from ~\cite[7.2]{BK1}. 

That $(\J_\M,\l_\M)$ is a cover of $(\J_\LL,\l_\LL)$ follows from
Proposition~\ref{simplecase}, as does the second Hecke algebra
isomorphism.
\end{proof}


\subsection{The general case}

We now treat the general case, where the ps-characters $\T_j$ are
not all endo-equivalent. We define an equivalence relation $\sim$ on
$\{1,\ldots,r\}$ by 
$$
j\sim k\ \iff\ \T_j,\T_k\hbox{ are endo-equivalent},
$$
and denote by $I_1,\ldots,I_l$ the equivalence classes. We put
$\Y^{(i)}=\bigoplus_{j\in I_i}\V^j$. As in the homogeneous case, we
assume that, for fixed $i$, every $j\in I_i$ has the same
ps-character, defined relative to the same simple pair. We write
$\La^{(i)}$ for an $\Oo_\D$-lattice sequence in
$\Y^{(i)}$, as in the homogeneous case. 
By changing in their affine class, we may (and do) assume that all the
$\La^{(i)}$ have 
the same $\Oo_\F$-period; as in the homogeneous case, we suppose also
that we have replaced the lattice sequences $\La^j$ with sequences in
their affine class, so that $\La^{(i)}=\bigoplus_{j\in I_i}\La^j$.

We write $\M$ for the Levi subgroup of $\G$ which stabilizes the
decomposition $\V=\bigoplus_{i=1}^l \Y^{(i)}$, and $(\J_\M,\l_\M)$ for
the cover of $(\J_\LL,\l_\LL)$ in $\M$, given by the homogeneous case
in Proposition~\ref{prop:hom}. Note that this $\M$ is now \emph{not}
the Levi subgroup of the introduction, but one rather larger.

We put $\La=\bigoplus_{i=1}^l \La^{(i)}$, a (not necessarily
strict) $\Oo_\D$-lattice sequence in $\V$,
and $\b=\sum_{j=1}^r \b_j\in\A=\End_\D(\V)$. Then $[\La,n,0,\b]$ is a
(non-simple) stratum in $\A$, for a suitable integer $n$.  

For $0\le t\le n$, we write $\T_j^{(t)}$ for the ps-character
defined by the character $\t_j|\H^{t+1}(\beta_j,\La^j)$.

\begin{theo}[{cf.~\cite[Main~Theorem,~p.94]{BKsemi}}]\label{thm:general}
There is a cover $(\K,\tau)$ of the type $(\J_\M,\l_\M)$ with the
following properties:
\begin{enumerate}
\item[(i)] $\U^{n+1}(\La)\subseteq \K\subseteq \U(\La)$;
\item[(ii)] if the ps-characters $\T_j^{(t)}$, for $1\le j\le r$, are
  endo-equivalent, and $([\La,n,0,\gamma],\vartheta,t)$ is a common
  approximation of $(\t_1,\ldots,\t_r)$, then
\begin{enumerate}
\item $\K$ contains and normalizes $\H^{t+1}(\gamma,\La)\cdot
  \left(\H^t(\gamma,\La)\cap\M\right)$; 
\item $\tau|\H^{t+1}(\gamma,\La)$ is a multiple of $\vartheta$;
\item $\tau|\H^t(\gamma,\La)\cap\LL$ is a multiple of
  $\t_1\otimes\cdots\otimes\t_r$;
\end{enumerate}
\item[(iii)] there is a support-preserving isomorphism of Hecke algebras
  $\Hh(\M,\l_\M)\simeq\Hh(\G,\tau)$. 
\end{enumerate}
\end{theo}

\begin{proof} The proof is by induction on $r$, the case $r=1$ being
empty. So let $r>1$ and suppose that $t\ge 0$ is minimal such that the
ps-characters $\T_j^{(t)}$ are endo-equivalent. Let
$([\La,n,0,\gamma],\vartheta,t)$ be a common approximation of
$(\t_1,\ldots,\t_r)$. If $t=0$ then the Theorem is given by
Proposition~\ref{prop:hom}, so we assume $t>0$. We allow $t=n$ (that is, the
ps-characters $\T_j^{(n-1)}$ are not all endo-equivalent) in
which case $\vartheta$ is the trivial character of $\U^{n+1}(\La)$. We
use the notation of~\S\ref{S.common}.

For $1\le j\le r$, let $c_j\in\PP_{-t}(\La^j)$ be such that 
$\t_j|\H^{t+1}(\beta_j,\La^j) =
\vartheta\psi_{c_j}$. By Corollary~\ref{cor:derived}, the derived stratum
$[\Ga^j_\g,t,t-1,s_\g(c_j)]$ is equivalent to a simple (or null)
stratum and we write $\phi_j(X)$ for the minimum polynomial of this
stratum (so that the characteristic polynomial is a power of
$\phi_j(X)$). 
We define an equivalence relation on $\{1,\ldots,r\}$ by 
$$
j\sim_t k\ \iff\ 
\T^{(t-1)}_j,\T^{(t-1)}_k\hbox{ are endo-equivalent}.
$$
Note that, by Corollary~\ref{cor:endocommon}, we have
$j\sim_t k$ if and only if $\phi_j(X)=\phi_k(X)$.
Let $J$ denote an equivalence class for $\sim_t$ for which the
minimum polynomial is not $X$; then $J$ is a union of
certain equivalence classes $I_i$ but is not the whole of $\{1,\ldots,r\}$,
or else we would contradict the minimality of $t$. 

Set
$\Z=\bigoplus_{j\in J} \V^j$ and
$\Z'=\bigoplus_{j\not\in J}\V^j$; let $\bar\M$ be the Levi subgroup
which stabilizes the decomposition $\V=\Z\oplus\Z'$ and let
$\bar\P=\bar\M\bar\N$ be a parabolic subgroup with Levi component
$\bar\M$ and opposite $\bar\P^-=\bar\M\bar\N^-$. By the inductive
hypothesis, we have a cover $(\K_{\bar\M},\tau_{\bar\M})$ of
$(\K_{\M},\tau_{\M})$ satisfying the conditions of the theorem (with
$\bar\M$ in place of $\G$). We define the group $\K$ by
$$
\K = \K_{\bar\M} \H^t(\g,\La)
  \cdot (\U^1(\La)\cap\B)\Om_{q-t+1}(\g,\La)\cap\bar\N,
$$
where $\Om_{q-t+1}(\g,\La)$ is the group defined in~\cite[\S2.8]{SS4}.
Then~\cite[Corollaire~4.6]{SS4} says that there is a unique
irreducible representation $\tau$ of $\K$ such that $(\K,\tau)$ is a
cover of $(\K_{\bar\M},\tau_{\bar\M})$, which has all the required
properties by transitivity of covers. [Note that, although it is
assumed in~\cite[\S4]{SS4} that the lattice sequence $\La$ is strict,
this extra condition is never used.]
\end{proof}

Now Theorem~\ref{thm:coverhecke} of the introduction follows from
Theorem~\ref{thm:general} and Proposition~\ref{prop:hom}, whence the
Main Theorem.

\providecommand{\bysame}{\leavevmode ---\ }
\providecommand{\og}{``}
\providecommand{\fg}{''}
\providecommand{\smfandname}{\&}
\providecommand{\smfedsname}{\'eds.}
\providecommand{\smfedname}{\'ed.}
\providecommand{\smfmastersthesisname}{M\'emoire}
\providecommand{\smfphdthesisname}{Th\`ese}

\end{document}